\documentclass[11pt]{amsart}
\pdfoutput=1
\usepackage{amsmath, amssymb}
\usepackage{amsfonts}
\usepackage{graphicx}
\usepackage{amsthm}
\usepackage{float}
\usepackage{amsthm}
\usepackage{breqn}
\usepackage{mathtools}
\usepackage{verbatim}
\usepackage[margin=.85 in]{geometry}
\allowdisplaybreaks

\usepackage{hyperref}
\hypersetup{pdfauthor={},pdftitle={combinatorially random curves on surfaces},colorlinks=true,linkcolor=black,citecolor=black}

\newtheorem{theorem}{Theorem}[section]
\newtheorem{lemma}[theorem]{Lemma}
\newtheorem{proposition}[theorem]{Proposition}
\newtheorem{corollary}[theorem]{Corollary}
\newtheorem{conjecture}[theorem]{Conjecture}

\newtheorem*{theorem*}{Theorem}
\theoremstyle{remark}
\newtheorem{remark}[theorem]{Remark}

\newtheorem{question}{Question}

\numberwithin{equation}{section}

\newcommand{\bP}{\mathbb{P}}

\newcommand{\bR}{\mathbb{R}}

\newcommand{\cA}{\mathcal{A}}
\newcommand{\cT}{\mathcal{T}}
\newcommand{\Cay}{\mathrm{Cay}}      

\begin{document}
\title[combinatorially random curves on surfaces]{Combinatorially random curves on surfaces}
\author{Tarik Aougab}
\address{Department of Mathematics \\ Haverford College \\ Haverford, PA 19041}
\email{taougab@haverford.edu}

\author{Jonah Gaster}
\address{Department of mathematical sciences \\ University of Wisconsin-Milwaukee \\ Milwaukee, WI 53211} 
\email{gaster@uwm.edu}

\date{\today}

\begin{abstract}

We study topological properties of random closed curves on an orientable surface $S$ of negative Euler characteristic. Letting $\gamma_{n}$ denote the conjugacy class of the $n^{th}$ step of a simple random walk on the Cayley graph driven by a measure whose support is a finite generating set, then with probability converging to $1$ as $n$ goes to infinity,
\begin{enumerate}
\item  the point in Teichm{\"u}ller space at which $\gamma_{n}$ is length-minimized does not escape to infinity when $S$ is closed; 
\item the self-intersection number of $\gamma_{n}$ is on the order of $n^{2}$, the minimum length of $\gamma_{n}$ taken over all hyperbolic metrics is on the order of $n$, and the metric minimizing the length of $\gamma_{n}$ is uniformly thick;
\item when $S$ is punctured and the distribution is uniform and supported on a generating set of minimum size, the minimum degree of a cover to which $\gamma_{n}$ admits a simple elevation (which we call the \textit{simple lifting degree} of $\gamma_{n}$) grows at least like $n/\log(n)$ and at most on the order of $n$.
\end{enumerate}
We also show that many of these properties are \textit{generic}, in the sense that the proportion of elements in the ball of radius $n$ in the Cayley graph for which they hold, converges to $1$ as $n$ goes to infinity. 

The lower bounds on simple lifting degree for randomly chosen curves we obtain significantly improve the previously best known bounds which were on the order of $\log^{(1/3)}n$. As applications, we give relatively sharp upper and lower bounds on the dilatation of a generic point-pushing pseudo-Anosov homeomorphism in terms of the self-intersection number of its defining curve, as well as upper bounds on the simple lifting degree of a random curve in terms of its intersection number which outperform bounds for general curves. 

\end{abstract}

\maketitle

\section{Introduction}

Let $S$ be a closed oriented surface of genus at least $2$ and let $\cT(S)$ indicate the Teichm\"uller space of marked hyperbolic structures on $S$. 
Let $\gamma$ be the free homotopy class of an immersed curve on $S$, and consider $\ell(\gamma):\cT(S)\to \bR$, the hyperbolic length function of $\gamma$. 
We are interested in the quantity $L(\gamma)= \inf \ell(\gamma)$, where the infimum is taken over $\cT(S)$ and, when it exists, the minimizing point $X(\gamma)\in \cT(S)$.
The aim of this paper is to study the quantities $L(\gamma)$, $X(\gamma)$ and their relationship to other invariants of curves, such as the self-intersection $\iota(\gamma,\gamma)$ or the simple lifting degree $\deg(\gamma)$ (defined to be the minimum degree of a finite cover $Y \rightarrow S$ so that $\gamma$ admits a simple elevation), when $\gamma$ is chosen ``at random''.

One version of the notion of a ``random curve'', well-studied in the literature, can be constructed by first choosing a marked hyperbolic metric $\sigma$, viewed as a point in $\cT(S)$. 
Because the unit tangent bundle of a finite-volume hyperbolic surface is itself finite-volume, one can choose a unit tangent vector on $S$ uniformly at random. 
Following this basepoint with the geodesic flow for a long time, one builds ``hyperbolically random'', or $\sigma-$\textit{random} curves $\gamma_T$ on $S$, roughly of $\sigma$-length $T$. 
In fact, an appropriate weighting of the resulting curves limit to the Liouville current of the hyperbolic surface.
This makes the quantities $L(\gamma_T)$, $\rho(\gamma_T)$, and the self-intersection $\iota(\gamma_T,\gamma_T)$ all fairly transparent. %\footnote{$\deg(\gamma)$ remains somewhat opaque; upper bounds are of the form $\sqrt{\iota(\gamma,\gamma)}$ while available lower bounds are of the form $\log \iota(\gamma,\gamma)$)}
Since the Liouville current is length-minimized at $\sigma$, the length minimizing metric of $\gamma_T$ converges to $\sigma$: evidently the `random' story has been biased by the choice at the outset of a hyperbolic metric.

The aim of this paper is to study the situation when $\gamma$ is instead chosen in a combinatorially random fashion. Let $\cA=\{w_1,\ldots,w_{k}\}$ be a finite generating set for $\pi_1S$, let $\Cay_\cA(S)$ indicate the Cayley graph of $\pi_1S$ with respect to $\cA$, and let $\mu$ be a non-elementary probability distribution supported on $\cA$. Let $\gamma_n$ be the (unoriented conjugacy class) of the result of an $n$-step random walk on $\Cay_\cA(S)$ driven by $\mu$.

By $\bP_{w}$, ($w$ for \textit{walk}) we will mean the probability operator induced by a simple random walk on $\Cay_\cA(S)$, and by $\bP^{(n)}_{g}$ ($g$ for \textit{generic}), we will mean the probability operator associated to the uniform distribution on the ball of radius $n$ about the identity in $\Cay_\cA(S)$. 

We show that, in many ways, the geometric properties of a combinatorially random curve mirror those of a $\sigma$-random curve. In the setting of closed surfaces, we arrive at these properties by first analyzing the location in $\cT(S)$ of the length-minimizing metric for $\gamma_{n}$:

\begin{theorem} \label{thm: compact} 

Assume $S$ is closed. Then for almost every sample path $\left\{w_{n} \right\}$, there is a compact set $K \subset \cT(S)$ containing $X(\gamma_{n})$ for all $n$. 

Moreover, if $S$ is any surface of finite type, there is a universal constant $\epsilon >0$ so that 
\[  \bP_{w} \left[ X(\gamma_{n}) \hspace{2 mm} \mbox{is} \hspace{2 mm} \epsilon-\mbox{thick}  \right] \xrightarrow[\substack{n\to\infty}]{} 1, \]
and 
\[  \bP_{w} \left[ \ell_{X(\gamma_{n})}(\gamma_{n})  \ge \epsilon \cdot n  \right] \xrightarrow[\substack{n\to\infty}]{} 1. \]
The analogous statements also hold for $\bP_{g}$: the $\bP^{(n)}_{g}$ probability that $X(\gamma)$ is $\epsilon$-thick and that $\ell_{X(\gamma)}$ is at least $\epsilon \cdot n$, goes to $1$ as $n \rightarrow \infty$. 
\end{theorem}

In the setting of closed surfaces, one can use Theorem \ref{thm: compact} to  prove that the self-intersection number of a combinatorially random curve grows quadratically in the length of the walk. We also give an argument independent of Theorem \ref{thm: compact} which applies to all surfaces of finite type and which yield uniform quadratic lower bounds on intersection number, and thus we have:

\begin{theorem} \label{thm: intersection}

If $S$ is any finite type surface, then there is $L \ge 1$ so that 

\[\bP_{w} \left[ \frac{n^{2}}{L} \leq i(w_n,w_n ) \leq L \cdot n^{2}  \right] \xrightarrow[\substack{n\to\infty}]{} 1. \]

Moreover,  for $\gamma \in B_{n}$, 
 \[ \bP^{(n)}_{g} \left[ \frac{n^{2}}{L} \leq i(\gamma, \gamma) \leq L \cdot n^{2}  \right]  \xrightarrow[\substack{n\to\infty}]{} 1.\]

\end{theorem}

We point out that, in simplest terms, the random walk model counts \textit{paths} in the Cayley graph, whereas choosing uniformly at random in larger and larger balls counts \textit{elements}. Depending on the group, relationships may exist between the two models, but there is in general no straightforward way of converting from one to the other. 

\subsection{Simple lifting degree}

Given $\gamma \in \pi_{1}(S)$ and a finite degree cover $p:\Sigma \rightarrow S$, an \textit{elevation} of $\gamma$ to $\Sigma$ is a lift $\tilde{\gamma} \in \pi_{1}(\Sigma)$ of some power of $\gamma$ (see preliminaries for a more formal definition).

The \textit{simple lifting degree} of $\gamma$, denoted $\deg(\gamma)$, is the minimum $n$ such that there exists a degree $n$ cover $p: \Sigma \rightarrow S$ and an elevation $\tilde{\gamma}$ to $\Sigma$ such that $\tilde{\gamma}$ is a simple closed curve. A classical and celebrated result of Scott (\cite{Scott78}, \cite{Scott782}) implies that $\deg(\gamma)$ is always finite and thus well-defined. 

Simple lifting degree has been studied extensively, see for instance \cite{Patel14}, \cite{AGPS17}, \cite{Gaster16}, \cite{AC19}. For example, Patel shows that $\deg(\gamma)$ can be bounded above by a linear function of the length of a geodesic representative for $\gamma$ on a particular hyperbolic metric \cite{Patel14}. Aougab-Gaster-Patel-Sapir show that $\deg(\gamma)$ is bounded above by some linear function of the self-intersection number \cite{AGPS17} (with constants depending on the complexity of the underlying surface), and Arenas--Coto-Neumann improve this to the uniform estimate $\deg(\gamma)\le 5i(\gamma,\gamma)+5$ \cite[Thm.~1.1]{AC19}. 

In the other direction, Gaster gives examples of curves $\gamma$ such that $\deg(\gamma)$ is essentially $i(\gamma, \gamma)$ \cite{Gaster16}. However, it remains quite difficult to bound $\deg(\gamma)$ from below, especially in the generic setting. Gupta-Kapovich show that in the presence of punctures, with $\mathbb{P}_{w}$ probability converging to $1$, $\deg(\gamma_{n})$ grows faster than $\log(n)^{(1/3)}$ \cite{GuKa19}. We improve this lower bound to one that is almost linear in the length of the walk, and as above, our results hold for both random walks and generically chosen curves. 

\begin{theorem} \label{thm: simple lifting degree for punctures} 

If $S$ is punctured, $\cA$ is of minimum size, and $\mu$ is the uniformy distribution on $\cA$, then there is $\mathcal{M}>0$ so that 
\[ \bP_{w} \left[ \frac{1}{\mathcal{M}} \cdot \frac{n}{\log(n)} \leq \deg(\gamma_{n}) \leq \mathcal{M} \cdot n \right] \xrightarrow[\substack{n\to\infty}]{} 1. \]

Moreover, for $\cA$ any finite generating set and $B_{n}$ denoting the ball of radius $n$ in the associated Cayley graph, one has that for $\gamma \in B_{n}$,  
\[ \bP^{(n)}_{g} \left[ \frac{1}{\mathcal{M}} \cdot \frac{n}{\log(n)} \leq \deg(\gamma) \leq \mathcal{M} \cdot n \right] \xrightarrow[\substack{n\to\infty}]{} 1\]

\end{theorem}

When $S$ is closed, we prove a genericity result for lower bounds on lifting degree that pertains to conjugacy classes:%, or for points in an orbit associated to a discrete faithful representation into $PSL_{2}(\mathbb{R})$:

\begin{theorem} \label{thm: simple lifting degree for closed surfaces} 

Let $S$ be a closed surface, and let $C_{n}$ denote the set of all conjugacy classes admitting a representative of word length at most $n$ (with respect to any finite generating set). Then there exists some $\mathcal{Q} >0$
\[ \frac{\#( [c] \in C_{n} : \deg([c]) \geq \mathcal{Q} \cdot n/\log(n) )}{\#(C_{n})}  \xrightarrow[\substack{n\to\infty}]{} 1. \]

\begin{comment}
Finally, is $\rho: \pi_{1}(S) \rightarrow PSL_{2}(\mathbb{R})$ is discrete and faithful, $o \in \mathbb{H}^{2}$ is an arbitrary choice of basepoint, then 

\[ \lim_{L \rightarrow \infty} \frac{\#(\gamma \in \pi_{1}(S): d_{\mathbb{H}^{2}}(o, \rho(\gamma) \cdot o) \leq L \hspace{1 mm} \mbox{and} \hspace{1 mm} \deg(\gamma) \geq \frac{1}{\mathcal{Q}} \cdot n/\log(n))}{\#(\gamma \in \pi_{1}(S): d_{\mathbb{H}^{2}}(o, \rho(\gamma) \cdot o) \leq L)} = 1. \]
\end{comment}

\end{theorem}

As a corollary of the lower bounds on self-intersection in Theorem \ref{thm: intersection}, we prove that for ``random'' curves on a closed surface, the linear upper bounds on simple lifting degree from Arenas--Coto-Neumann \cite{AC19} or from Aougab-Gaster-Patel-Sapir \cite{AGPS17} in terms of intersection number can be improved to a square root upper bound: 

\begin{corollary} \label{better degree versus intersection} When $S$ is closed, there is some $J \geq 1$ so that 
\[ \bP_{w} \left[ \deg(\gamma_{n}) <   J \cdot \sqrt{i(\gamma_{n}, \gamma_{n})} \right] \xrightarrow[\substack{n\to\infty}]{} 1, \]
and similarly for $\gamma \in B_{n}$, 
\[\bP_{g}\left[ \deg(\gamma) <   J \cdot \sqrt{i(\gamma, \gamma)} \right] \xrightarrow[\substack{n\to\infty}]{} 1.  \]
\end{corollary}

\begin{proof} It is easy to see that any conjugacy class represented by a word of length at most $n$ has simple lifting degree at most on the order of $n$ (see the beginning of Section \ref{section: lifting} for the details). On the other hand, with $\bP_{w}$ (resp. $\bP_{g}$) probability converging to $1$ as $n \rightarrow \infty$, $\gamma_{n}$ (resp. $\gamma \in B_{n}$) has self intersection number at least on the order of $n^{2}$. The result follows. 
\end{proof}

\subsection{Random point-pushing maps} 

In the setting where $S= (S,p)$ has a preferred marked point or puncture $p$, one has the celebrated Birman exact sequence of mapping class groups associated to forgetting the puncture:
\[ 1 \rightarrow \pi_{1}(S,p) \rightarrow \mathcal{MCG}(S,p) \xrightarrow[]{\pi} \mathcal{MCG}(S) \rightarrow 1. \]
The kernel of $\pi$ is isomorphic to $\pi_{1}(S,p)$ and is called the \textit{point-pushing subgroup} of $\mathcal{MCG}(S,p)$ (see the preliminaries section for more details). 

We are motivated by the following general question: 
\begin{question} \label{comb point push}
Let $\gamma \in \pi_{1}(S,p)$ and suppose the corresponding point-pushing mapping class $f_{\gamma}$ is pseudo-Anosov. How can one relate dynamical properties of $f_{\gamma}$ (such as its translation length in Teichm{\"u}ller space or in the curve complex) to combinatorial or topological properties of the underlying curve $\gamma$? 
\end{question}

Question \ref{comb point push} seems quite difficult in general. As of now, the deepest result addressing it is due to Dowdall \cite{Dowdall11} who relates the dilatation $\mbox{dil}(f_{\gamma})$ of $f_{\gamma}$ to the self-intersection number of $\gamma$ as follows: 
\[ (i(\gamma, \gamma)+ 1)^{(1/5)} \leq \mbox{dil}(f_{\gamma}) \leq 9^{i(\gamma, \gamma)}. \]

Using Theorems \ref{thm: compact} and \ref{thm: intersection}, we can drastically improve these bounds for random or generic point-pushing maps: 

\begin{theorem} \label{thm: point push} 
There is $R>0$ so that 
\[ \bP_{w} \left[ e^{\sqrt{\frac{i(\gamma_{n}, \gamma_{n})}{R}}} \leq \mbox{dil}(f_{\gamma_{n}}) \leq e^{\sqrt{R \cdot i(\gamma_{n}, \gamma_{n})}} \right] \xrightarrow[\substack{n\to\infty}]{} 1.  \]

Similarly, for $\gamma \in B_{n}$, 

\[ \bP^{(n)}_{g}\left[ e^{\sqrt{\frac{i(\gamma, \gamma)}{R}}} \leq \mbox{dil}(f_{\gamma}) \leq e^{\sqrt{R \cdot i(\gamma, \gamma)}} \right] \xrightarrow[\substack{n\to\infty}]{} 1. \]

\end{theorem}

\subsection{Questions and conjectures}

\subsubsection{Locating length minimizers} In analogy with the way in which properties of a $\sigma$-random curve are ``biased'' by the choice of hyperbolic metric, we conjecture that geometric and topological properties of $\gamma_{n}$ are biased by the choice of generating set. We formalize this by conjecturing a strengthening of Theorem \ref{thm: compact} which states that $X(\gamma_{n})$ remains uniformly close to the point $X_{\cA}$ minimizing the hyperbolic length of the rose for the generating set $\cA$: 

\begin{conjecture} \label{conj: compact} There is some $D>0$ so that 
\[ \bP_{w} \left[ d_{\cT}(X(\gamma_{n}), X_{\cA}) \le D \right] \xrightarrow[\substack{n\to\infty}]{} 1,\]
where $d_{\cT}$ denotes distance in the Teichm\"{u}ller metric. Similarly, for $\gamma \in B_{n}$, 

\[ \bP_{g}^{(n)}\left[ d_{\cT}(X(\gamma), X_{\cA}) \le D \right] \xrightarrow[\substack{n\to\infty}]{} 1. \]

\end{conjecture}

Theorem \ref{thm: compact} states that for almost every sample path, $X(\gamma_{n})$ remains some uniformly bounded distance away from $X_{\cA}$, but this distance can a priori grow as one moves from sample path to sample path. 

\subsubsection{Counting in a conjugacy class} The strategy for proving the lower bounds on simple lifting degree in Theorem \ref{thm: simple lifting degree for punctures} is to first prove the sort of genericity result for conjugacy classes stated in Theorem \ref{thm: simple lifting degree for closed surfaces} (see Section \ref{strategy} for a brief description of the idea); indeed, when proving this genericity statement for conjugacy classes, we handle both closed and punctured surfaces simultaneously. The results in Theorem \ref{thm: simple lifting degree for punctures} then follow by giving an upper bound which has a strictly smaller exponential growth rate than the group itself for the intersection of any conjugacy class with a ball of radius $n$ in the Cayley graph (see Lemma \ref{lem: bounding conjugacy}). Let $\mathfrak{k}_{n}$ denote this quantity, i.e., the maximum --taken over all conjugacy classes $\mathfrak{c}$-- of the size of the intersection $\mathfrak{c} \cap B_{n}$. 

We conjecture that similar bounds on $\mathfrak{k}_{n}$ hold in the setting of a surface group with a standard generating set; indeed, it is not too hard to see that in any hyperbolic group, the exponential growth rate of a fixed conjugacy class is always half the growth of the parent group (see for instance \cite{PP15} for a proof). Proofs of this fact rely on standard properties of hyperbolicity, for example the fact that the centralizer of an element is quasi-convex. Crucially, these arguments apply in the regime when a conjugacy class is fixed once and for all. In general, the quality of the quasi-convexity of a centralizer $C_{\gamma}$ depends on the word length of $\gamma$, so obtaining bounds on the \textit{maximum} possible size-- over all conjugacy classes-- of $\mathfrak{c} \cap B_{n}$ seems like a potentially more delicate problem in general. 

This motivates the following two conjectures; when combined with the proofs in Section \ref{section: lifting}, verifying them would imply a version of Theorem \ref{thm: simple lifting degree for punctures} for closed surfaces: 

\begin{conjecture} Let $\Gamma_{g}$ denote the Cayley graph of a closed surface group with respect to a standard generating set. Then there is some polynomial $p$ so that 

\[ \max_{\mathfrak{c}} \#(\mathfrak{c} \cap B_{n}) \leq p(||\mathfrak{c}||) \cdot \#(B_{(n- ||\mathfrak{c}||)/2}), \]
where $||\mathfrak{c}||$ denotes the minimum word length of a representative of the conjugacy class $\mathfrak{c}$. 

\end{conjecture}

\begin{conjecture} \label{maximizing over conjugacy} Let $\Gamma_{g}$ be as above. Then there is some constant $\rho < 1$ so that
\[ \max_{\mathfrak{c}}\bP_{w}(w_{n} \in \mathfrak{c}) \leq \rho^{n}. \]
\end{conjecture}

We think of both of these conjectures as an attempt to apply statements already known for individual conjugacy classes, to \textit{all} conjugacy classes at once and in a uniform way. Indeed, as mentioned above, the exponential growth rate of any individual conjugacy class is understood. Analogously, using work of Maher-Tiozzo outlined below in Section \ref{subsec: R}, one can show that the $\bP_{w}$ probability that $w_{n}$ lies in some \textit{fixed} conjugacy class decays exponentially fast with $n$. In fact, we hypothesize that Conjecture \ref{maximizing over conjugacy} holds in any hyperbolic group when the distribution $\mu$ is supported on some finite generating set.

\subsubsection{Criteria for large lifting degree} It is in general quite difficult to bound the simple lifting degree from below in terms of explicit combinatorial properties of a given curve $\gamma$. The aforementioned work of the second author \cite{Gaster16} represents one of the only results in this spirit. Loosely speaking, it says that if a curve $\gamma$ enters some annulus $A$ on the surface, spirals $k$ times around its core, and then exits $A$ over the same boundary component it crossed to enter, then $\deg(\gamma) \ge k$ (see Lemma \ref{lem:degree} in Section \ref{section: lifting} for a formal statement and proof). If $\gamma$ exhibits such behavior, we will call $A$ a $k$-\textit{spiraling annulus} for $\gamma$. 

One can use work of Sisto-Taylor \cite{ST19} to show that with $\bP_{w}$ probability converging to $1$, a given annulus $A$ can be  a $k$-spiralling annulus for $\gamma$ only if $k$ is at most on the order of $\log(n)$ (again, see Section \ref{section: lifting} for details). It therefore follows by Theorems \ref{thm: simple lifting degree for punctures} and \ref{thm: simple lifting degree for closed surfaces} that there are many curves whose lifting degree is large \textit{for some other reason besides the presence of spiraling annuli}.

This motivates the search for a deeper understanding of how to translate between lifting degree and combinatorics: 

\begin{question} Are there simple combinatorial properties exhibited by a generic curve $\gamma$ which force $\deg(\gamma)$ to be large?
\end{question}

\subsection{Outline of strategy and tools} \label{strategy}

To prove Theorems \ref{thm: compact} and \ref{thm: intersection}, we combine Teicm{\"u}ller theoretic arguments with probabilistic and counting results in the setting of negative curvature, due to Gekhtman-Taylor-Tiozzo \cite{GTT18} and Maher-Tiozzo \cite{MT18}. 

To get a feel for the strategy, we briefly outline the proof of Theorem \ref{thm: compact} for random walks. Let $X_{n} \in \cT(S)$ denote the metric minimizing the hyperbolic length of $\gamma_{n}$; our goal is to show that $X_{n}$ lies in some fixed neighborhood of $X_{\cA}$. Since $\gamma_{n}$ has length on the order of $n$ on the surface $X_{\cA}$, it follows that the sequence of geodesic currents given by $\left\{w_{n}/n \right\}$ is pre-compact and thus there exists a limiting current $c$.

We use Maher-Tiozzo \cite{MT18} to show that with high probability, $\gamma_{n}$ intersects any given simple closed curve at least linearly in $n$ many times. We use this to bound from below the infimal intersection number between $c$ and any simple closed curve, and in the process prove that $c$ must be a filling current. It then follows that $c$ is length minimized at a unique interior point of $\cT(S)$, which-- using basic properties of length functions and the intersection pairing-- precludes the minimizing metrics $X_{n}$ from getting further and further away from $X_{\cA}$. These arguments occur in Section \ref{section: minimizer}. 

With Theorem \ref{thm: compact}, we approach Theorem \ref{thm: intersection} as follows. Let $\ell_{n}$ denote the geodesic length of $\gamma_{n}$ on its length-minimizing surface $X_{n}$, and consider the weighted curve $\gamma_{n}/\ell_{n}$, viewed as a geodesic current on $X_{\gamma_{n}}$. Again, work of Bonahon \cite{Bon88} implies the existence of an accumulation point $c$ of the sequence $( \gamma_{n} / \ell_{n} )$. 

Since, by Theorem \ref{thm: compact}, the sequence of points $( X_{\gamma_{n}} )$ all lie in a compact subset of $\cT(S)$, the geodesic current $c$ must be length-minimized somewhere in the interior of $\cT(S)$. On the other hand, if $i(\gamma_{n}, \gamma_{n})$ grows sub-quadratically, one can show that the self-intersection number of $\gamma_{n}/\ell_{n}$ goes to $0$ as $n \rightarrow \infty$. It would then follow that $c$ is a geodesic lamination, and no geodesic lamination is length-minimized at an interior point of $\cT(S)$. If $S$ is any finite type surface,  work of Sapir \cite{S16} implies that if the self-intersection number is at most $k$, there must exist a simple closed curve intersecting $\gamma_{n}$ at most $\sqrt{k}$ times. These arguments can be found in Section \ref{section: intersection}. 

To prove Theorems \ref{thm: simple lifting degree for punctures} and \ref{thm: simple lifting degree for closed surfaces}, we use arguments that rely on counting simple closed geodesics on a hyperbolic surface up to a given length, in the spirit of Mirzakhani \cite{M08}-- we quickly summarize the idea for proving the genericity result for conjugacy classes.

Let $d =\deg(\gamma)$, let $\Sigma$ be the degree $d$ cover of $S$ such that $\gamma$ admits a simple elevation to $\Sigma$, and let $\tilde{\gamma}$ denote this elevation. Equip $S$ with a complete hyperbolic metric and let $\ell(\gamma)$ denote the geodesic length of $\gamma$ in this metric. Since $\gamma$ represents a word that lives in $B_{n}$, the Milnor-Svarc lemma implies that $\ell(\gamma)$ is $O(n)$. The chosen metric pulls back to $\Sigma$, and by basic covering space theory, $\tilde{\gamma}$ admits a representative of length at most $d \cdot \ell(\gamma)$. 

Thus, the number of conjugacy classes $[\gamma]$ in $C_{n}$ admitting an elevation to a cover of degree at most $d$ is bounded above by the number of conjugacy classes of degree $d$ covers of $S$, multiplied by the number of simple closed geodesics on any such cover with length at most $d \ell(\gamma)$. 

Using Dehn-Thurston coordinates in a careful way, we estimate the number of all such curves from above, and show that they grow sub-exponentially in $n$ under the assumption that $d$ grows slower than $n/\log(n)$. The details can be found in Section \ref{section: lifting}. Since our main results for $\bP_{w}$ in this section are stated only for standard generating sets, we conclude this section with a way to obtain lower bounds on $\deg(w_{n})$ that hold with high $\bP_{w}$ probability for any finite generating set $S$, and with any distribution whose support is $S$. However, these more general lower bounds are only on the order of $\log(n)$. 

Finally, we conclude the paper with a short proof of Theorem \ref{thm: point push} in Section \ref{section: point push}. The idea is simple: one uses either Maher-Tiozzo \cite{MT18} (in the setting of $\bP_{w}$) or Gekhtman-Taylor-Tiozzo \cite{GTT18} (in the setting of $\bP_{g}$) to argue that $\mbox{dil}(f_{\gamma_{n}})$ (or respectively $\mbox{dil}(f_{\gamma})$ for some $\gamma \in B_{n}$) grows exponentially in $n$ with high probability. On the other hand, Theorem \ref{thm: intersection} implies that $i(\gamma_{n}, \gamma_{n})$ grows quadratically in $n$, and so the theorem follows by bounding $\mbox{dil}(f_{\gamma_{n}})$ above and below by functions of $n$, and then expressing $n$ as a function of $i(\gamma_{n}, \gamma_{n})$.

\subsection{Acknowledgements.} The authors thank Ilya Gekhtman and Mark Hagen for guidance regarding the growth of conjugacy classes (and the growth of a single conjugacy class) relevant in Section \ref{section: lifting}. The authors also thank Samuel Taylor for many extremely helpful conversations. The first author was partially supported by NSF grant DMS 1939936. 

\section{Preliminaries} \label{section: prelim} 

\subsection{Curves and surfaces} \label{subsec: CS}

In what follows, $S$ will be an orientable surface of finite type. By $S_{g,p,b}$ we will mean the surface of genus $g$ with $p \geq 0$ punctures and $b \geq 0$ boundary components.

A \textit{closed curve} is a continuous function $\gamma: S^{1} \rightarrow S$. It is \textit{essential} if it is not homotopic to a constant map or into an arbitrarily small neighborhood of a puncture. We will use the notation $\sim$ for homotopy. We will sometimes conflate a curve with its image, or its entire homotopy class, when convenient. A curve is called \textit{simple} if it is an embedding. We will also sometimes refer to an entire homotopy class as simple when it admits a simple representative. 

Given a closed curve $\gamma: S^{1} \rightarrow S$ and a finite degree covering $p: \Sigma \rightarrow S$, an \textit{elevation} of $\gamma$ to $\Sigma$ is a closed curve $\tilde{\gamma}: S^{1} \rightarrow \Sigma$ such that there exists a covering map $\rho : S^{1} \rightarrow S^{1}$ with $p \circ \tilde{\gamma} = \gamma \circ \rho$. In other words, if we view $\gamma \in \pi_{1}(S, x)$ as an element of the fundamental group for some $x \in \mbox{Im}(\gamma)$, then $\tilde{\gamma} \in \pi_{1}(\Sigma, \tilde{x})$ is a path lift of some power of $\gamma$. 

Given curves $\alpha, \beta$, their \textit{geometric intersection number}, denoted $i(\alpha, \beta)$, is the minimum set-theoretic intersection taken over all representatives of the two homotopy classes: 

\[ i(\alpha, \beta) = \min_{\alpha' \sim \alpha, \beta' \sim \beta}|\alpha' \cap \beta'|. \]

Curves $\alpha, \beta$ are said to be in \textit{minimal position} if they achieve the geometric intersection number for their respective homotopy classes. Similarly, a single curve $\alpha$ is in \textit{minimal position} if it achieves the geometric self-intersection number $i(\alpha, \alpha)$. 
%The formula above is wrong for self-intersection number

Given a collection of curves $\Gamma = \left\{\gamma_{1},..., \gamma_{n} \right\}$ in pairwise minimal position, $\Gamma$ is said to \textit{fill} $S$ if $S \setminus \Gamma$ is a union of disks or once-punctured disks. 

A \textit{pants decomposition} of $S$ is a collection $\mathcal{P}$ of essential and pairwise disjoint simple closed curves so that $S \setminus P$ is a disjoint union of pairs of pants (meaning a topological sphere such that the sum of boundary components and punctures is $3$). A standard Euler characteristic argument yields that the size of any pants decomposition is determined completely by the values of $g,p,b$. 

A \textit{multi-curve} is a finite collection of pairwise non-homotopic closed curves $\left\{\gamma_{1},..., \gamma_{n} \right\}$. A multi-curve is \textit{simple} if each component of it is simple and disjoint from all other components.

\subsection{Dehn-Thurston coordinates} \label{subsec: DT}

Fix a pants decomposition $\mathcal{P}= \left\{\gamma_{1},..., \gamma_{3g-3} \right\}$ of $S$ and let $\mathcal{S}(S)$ denote the set of homotopy classes of simple multi-curves on $S$. The \textit{Dehn-Thurston coordinates} is a parameterization of $\mathcal{S}(S)$ by coordinates in $\mathbb{Z}^{6g-6}$, described as follows (we follow the treatment of D. Thurston \cite{Thurston08} and Penner \cite{Penner}). 

The $i^{th}$ \textit{intersection number} $m_{i}$ of a given simple multi-curve $\alpha$ is simply the number of intersections between $\alpha$ and $\gamma_{i}$ when $\alpha$ is placed in minimal position with each $\gamma_{i}$. The  non-negative numbers $m_{1},..., m_{n}$ determine the intersection of $\alpha$ with each pair of pants in the complement of $\mathcal{P}$ up to isotopy. All that is left to do to recover $\alpha$ completely is to determine how neighboring pairs of pants glue together. For this, we need the \textit{twist numbers}, described informally as follows. 

Fix a ``doubling'' of each $\gamma_{i}$-- a choice of a parallel copy $\overline{\gamma_{i}}$ of $\gamma_{i}$. Each pair $\left\{\gamma_{i}, \overline{\gamma_{i}} \right\}$ bound an annulus $\mathcal{A}_{i}$. Letting $P$ be one of the pairs of pants in the complement of $\mathcal{P}$, let $P'$ be the complement of the intersection of $\bigcup_{i} \mathcal{A}_{i}$ with $P$. A basic lemma attributed to Dehn and Thurston (see for instance \cite{Penner}) implies that there is a way to isotope a given simple multi-curve $\alpha$ so that it intersects each $P'$ in a ``standard form'' that is determined by the isotopy class of the arc system $\alpha \cap P$. 

By choosing carefully a reference arc in each $A_{i}$, one can then define the $i^{th}$ twisting number to be the signed intersection number between this reference arc and $\alpha$.  Intuitively, it measures the number of times (and the direction in which) $\alpha$ twists about each pants curve when passing over it.

\begin{comment}

\subsection{Intersecting an immersed curve with an annulus} \label{subsec:beta} 

In \cite{AGPS17}, the authors together with Patel and Sapir (and following ideas from \cite{Sapir16}) outline a combinatorial encoding of a non-simple closed curve on $S$ in terms of how it interacts with a given pants decomposition $\mathcal{P}$. We refer the reader to either \cite{AGPS17} or \cite{Sapir16} for details, but we describe the essentials here very briefly. 

Given a closed curve $\alpha$ and a pants decomposition $\mathcal{P} =\left\{\gamma_{1},..., \gamma_{3g-3} \right\}$, one first fixes a choice of annulus $\mathcal{A}_{i}$ with core curve $\gamma_{i}$. Then $\alpha$ intersects $\mathcal{A}_{i}$ in potentially many arcs, each having one of two types: 
\begin{enumerate}
    \item $\beta$-\textit{arcs}, which have one endpoint on each boundary component of $\mathcal{A}_{i}$; and
    \item $\tau$-\textit{arcs}, which have both endpoints on one boundary component of $\mathcal{A}_{i}$. 
\end{enumerate}
Each arc of $\alpha$ in $\mathcal{A}_{i}$ -- whether it is a $\beta$ or a $\tau$ arc-- has a \textit{twisting number} which one can formalize in a way that is analogous to the twisting numbers of Dehn-Thurston coordinates. Intuitively, it measures the number of times the arc winds around $\gamma_{i}$ before leaving $\mathcal{A}_{i}$. Note that $\beta$-arcs are properly embedded, but $\tau$-arcs self-cross before leaving $\mathcal{A}_{i}$; moreover, only $\beta$-arcs intersect $\gamma_{i}$ essentially. 
\end{comment} 

\subsection{Hyperbolic geometry} \label{subsec: HG}

A \textit{hyperbolic metric} on $S$ is a complete, finite area Riemannian metric of constant sectional curvature $-1$. The surface $S$ equipped with such a metric will be called a \textit{hyperbolic surface}. 

It is a standard consequence of negative curvature that given a hyperbolic surface $S$, in every homotopy class of essential closed curves, there exists a unique length-minimizing representative. This representative is called the \textit{geodesic} in that homotopy class. Another standard consequence of negative curvature is that geodesics always realize the geometric intersection number. 
Given $\epsilon>0$, a hyperbolic surface is said to be $\epsilon$-\textit{thick} if every closed geodesic has length at least $\epsilon$. 

Given an essential curve $\gamma$, we can define its \textit{length function} 
\[ \ell_{\gamma}: \cT(S) \rightarrow \mathbb{R}, \]
which records the geodesic length of $\gamma$ at a point $X \in \cT(S)$. We indicate $L(\gamma)=\inf \ell_\gamma(X)$, the infimal geodesic length of $\gamma$ taken over all $X \in \cT(S)$. 

The \textit{collar lemma} states that any simple closed geodesic $\alpha$ on a hyperbolic surface has an embedded collar neighborhood whose width is inversely proportional to the length of $\alpha$ (see for instance \cite{Buser} for more details).

\begin{lemma} \label{lem: collar} Letting $\ell$ denote the geodesic length of a simple closed curve $\alpha$ on a hyperbolic surface $S$, there is an embedded collar neighborhood of $\alpha$ with width at least 
\[ \sinh^{-1}\left( \frac{1}{\sinh\left( \frac{\ell(\alpha)}{2} \right)} \right)~. \]

\end{lemma} 

The collar lemma implies the existence of some universal constant $C>0$ such that if $\alpha, \beta$ are two closed geodesics on a hyperbolic surface $S$ with geodesic lengths less than $C$, they must be disjoint.

In section \ref{section: lifting}, we will need to consider hyperbolic surfaces with geodesic boundary and perhaps also with finitely many punctures. When a surface of the form $S_{g,p,0}$ is equipped with an arbitrary complete hyperbolic metric, a classical theorem of Bers states the existence of a pants decomposition $\left\{\gamma_{1},..., \gamma_{3g-3+p} \right\}$ satisfying 
\[ \ell(\gamma_{k}) \leq 4k \cdot \log \left( \frac{4\pi (2g-2+p)}{k} \right)~.\]
See \cite{Buser} for more details. 

This was generalized to surfaces equipped with arbitrary complete Riemannian metrics with area normalized to be $4\pi \left( g + \frac{p}{2} - 1\right)$ on closed surfaces (where $p$ now denotes the number of marked points) by Balacheff-Parlier-Sabourau \cite{BPS12}.  They show that on such a surface, there is a pants decomposition $\mathcal{P}$ such that the total length $\ell(P)$ satisfies 
\[ \ell(P) \leq C_{g} \cdot p \log(p+1), \]
where $C_{g}$ depends only on genus. 

In fact, they prove a more general result which applies to surfaces $S_{g,p,b}$ with potentially multiple boundary components and with arbitrary finite area. Following through the proof of their Theorem $6.10$, one finds the existence of a polynomial $\mathcal{F}(x,y)$ which is of degree $4$ in $x$ and degree $3$ in $y$ such that if $S$ is a hyperbolic surface with totally geodesic boundary all of whose boundary components have length at most $L$, there is a pants decomposition of total length at most $\mathcal{F}(|\chi(S)|, L)$. 

\begin{comment}
following explicit bound on the length of any curve in the pants decomposition they construct as a function of $g,p,b, \text{Area}(S)$, and $L$, the maximum length of any boundary component (below, $C$ is a universal constant not depending on $S$):

\[ 2\cdot g^{5/2} \cdot C^{2} \cdot \left( \text{Area}(S) + \frac{b}{2\pi}L^{2} + \frac{2g \cdot \left(C^{2} \cdot g \cdot \text{Area}(S)+ \frac{b}{2\pi}L^{2} \right)}{2\pi} \right) \cdot \sqrt{\text{Area}(S)+ \frac{b}{2\pi}L^{2}} + b \cdot L \]

In the event that $S$ is hyperbolic, $\mbox{Area}(S)$ is of course a linear function of $g,p,$ and $b$. Thus the key take-away from the above is that there is an explicit polynomial $\mathcal{F}(x,y)$ which is of degree $4$ in $x$ and degree $3$ in $y$ such that if $S$ is a hyperbolic surface with totally geodesic boundary all of whose boundary components have length at most $L$, there is a pants decomposition of total length at most $\mathcal{F}(|\chi(S)|, L)$. 

\end{comment}

\begin{remark} \label{easier but worse} We remark that there is a more elementary way to construct a bounded length pants decomposition $\mathcal{P}$ on a hyperbolic surface $S$ with totally geodesic boundary, although the bounds one gets are not as strong as in \cite{BPS12}. We sketch this briefly:

First, add in any curve to $\mathcal{P}$ whose geodesic length is less than $C$. Cutting along all such curves yields a (possibly disconnected) hyperbolic surface $S'$ with boundary and which is $C$-thick. The original area of $S$ is bounded (in terms only of the topology of $S$) by the Gauss-Bonnet theorem, and so each component of $S'$ has uniformly bounded area. Thus the $C$-thick part of each component has diameter at most some $D$, bounded only in terms of the original topology of $S$. 

It follows that on each component $\Sigma$ of $S'$, there is a simple closed geodesic $\sigma$ of length at most $2(D+L)$, where $L$ is the maximum of $C$, the length of any original boundary component of $S$, and the length of the boundary of a collar neighborhood of a geodesic of length $C$. One simply starts at a boundary component of $\Sigma$ and picks the shortest essential arc with at least one endpoint there; complete this arc to a simple closed curve by concatenating it with arcs that run around pieces of the boundary. We then cut along $\sigma$, and repeat the argument. At each stage, the maximum length of a boundary component can (roughly) double, and the number of stages is bounded above linearly in terms of $|\chi(S)|$. As a result, some curves in $\mathcal{P}$ might have exponentially long length as a function of both $|\chi(S)|$ and $L$. 

For our purposes in Section \ref{section: lifting}, it will turn out that either this bound or the one from \cite{BPS12} will suffice. 

\end{remark}

%\[ A_{g,L} = \frac{C \cdot \sqrt{g} \cdot \sqrt{\text{Area}(N)}}{\sqrt{n}} \] 
%\[ \leq \frac{C \cdot \sqrt{g} \cdot \sqrt{\text{Area}(M)+ \frac{b}{2\pi}L^{2}}}{\sqrt{n}} \]
%\[ \leq C \cdot \sqrt{g} \cdot \sqrt{\text{Area}(M)+ \frac{b}{2\pi}L^{2}} \]

%\[ \text{Area}(M) + \frac{b}{2\pi}L^{2} + \frac{2gA_{g}^{2}n}{2\pi} \leq B_{g,L}n\]
%\[ \Rightarrow B_{g,L} \leq \text{Area}(M) + \frac{b}{2\pi}L^{2} + \frac{2g A^{2}_{g}}{2\pi} \]

%All and all, 
%\[ C_{g, L} \leq 2\cdot g \cdot C \cdot B_{g, L} \cdot g A_{g,L} + bL. \]

\subsection{Teichm{\"u}ller space and the mapping class group} \label{subsec: TS}

The \textit{Teichm{\"u}ller space} of $S$, denoted $\cT(S)$, is a space of (equivalence classes of) pairs $(\phi, \sigma)$ where $\sigma$ is a surface homeomorphic to $S$ equipped with a hyperbolic metric, and $\phi: S \rightarrow \sigma$ is a homeomorphism. Two pairs $(\phi, \sigma)$ and $(\phi', \sigma')$ are equivalent when there is an isometry $j: \sigma \rightarrow \sigma'$ such that $j \circ \phi = \phi'$ up to homotopy. The first coordinate is called the \textit{marking} of $X$. 

The \textit{mapping class group}, denoted $\mathcal{MCG}(S)$, is the group of homotopy classes of orientation preserving homeomorphisms of $S$. Given a finite set of marked points or punctures $\mathfrak{p} \subset S$, we denote by $\mathcal{MCG}(S, \mathfrak{p})$ the group or orientation preserving homeomorphisms sending $\mathfrak{p}$ to itself, up to homotopy. The mapping class group acts on $\cT(S)$ by composition with the marking. When $S$ has boundary, mapping classes correspond to homotopy classes of orientation preserving homeomorphisms that pointwise fix each boundary component. 

Mapping classes fit into a dynamical trichotomy known as the \textit{Nielsen-Thurston classification} (see \cite{FM12} for details): $f \in \mathcal{MCG}(S)$ is either finite order (known as elliptic); infinite order but preserving of some simple multi-curve (reducible); or there exists a pair of transverse measured filling (singular) foliations and a real number $\lambda>1$ such that both foliations are preserved by $f$ and such that one of the transverse measures is multiplied by $\lambda$ and the other by $1/\lambda$ (known as pseudo-Anosov). The \textit{dilatation} of a pseudo-Anosov homeomorphism, denoted $\mbox{dil}(f)$, is the number $\lambda$ (when $f \in \mathcal{MCG}(S)$ is not pseudo-Anosov, we will define $\mbox{dil}(f)$ to be zero by convention).

When a closed surface $S$ comes equipped with a preferred marked point $p$, the group $\mathcal{MCG}(S, p)$ fits into a short exact sequence known as the Birman exact sequence coming from the map $\pi: (S,p) \rightarrow S$ corresponding to forgetting the significance of $p$: 

\[ 1 \rightarrow \ker \pi \rightarrow \mathcal{MCG}(S,p) \xrightarrow[]{\pi} \mathcal{MCG}(S) \rightarrow 1. \]

The kernel of $\pi$ is naturally identified with the surface group $\pi_{1}(S,p)$ and is called the \textit{point-pushing subgroup} of $\mathcal{MCG}(S,p)$. An element $\gamma \in \pi_{1}(S,p)$ gives rise to point-pushing homeomorphism $f_{\gamma}$ by ``pushing'' the marked point $p$ around the curve $\gamma$ and dragging the surface along with it (see \cite{FM12} for formal details). 

A classical result of Kra \cite{Kra81} states that $f_{\gamma}$ is pseudo-Anosov exactly when $f$ fills $S$.

The Teichm{\"u}ller space can be topologized in several equivalent ways. One such way is to use so-called \textit{Fenchel-Nielsen coordinates} which in some sense mirrors the Dehn-Thurston coordinates mentioned above. Fixing a pants decomposition $\mathcal{P}= \left\{\gamma_{1},..., \gamma_{n}\right\}$, the Fenchel-Nielsen coordinates of a point $X \in \cT(S)$ are given by 
\[ \left( \ell_{\gamma_{1}}(X),..., \ell_{\gamma_{n}}(X), \tau_{1}(X),..., \tau_{n}(X) \right), \]
where $\ell_{\gamma_{i}}(X)$ is the geodesic length of $\gamma_{i}$ on $X$, and $\tau_{i}(X)$ is a \textit{twisting parameter} which measures the extent to which a reference arc crossing from one pair of pants to another twists around the bounding curve when the surface is glued together. See for instance \cite{Buser} for the formal details. 

One can then topologize $\cT(S)$ such that the Fenchel-Nielsen coordinates give a homeomorphism to $\mathbb{R}^{2n}$. With respect to this topology, the mapping class group acts properly discontinuously and by homeomorphisms. The quotient is naturally identified with the moduli space $\mathfrak{M}_{S}$ of hyperbolic surfaces homeomorphic to $S$. 

Given $\epsilon >0$, the $\epsilon$-\textit{thick} part $\cT_{\epsilon}(S)$ of $\cT(S)$ is the set of all $\epsilon$-thick hyperbolic surfaces. The projection of $\cT_{\epsilon}(S)$ to $\mathfrak{M}_{S}$ is compact. 

The Teichm{\"u}ller space admits a natural metric called the \textit{Teichm{\"u}ller metric}, which we will denote by $d_{\cT}$. The formal details can be found for instance in \cite{FM12}; the essential idea is that $d_{T}(X,Y)$ is the $\log$ of the infimal dilatation of a quasi-conformal homeomorphism from $X$ to $Y$, isotopic to the identity (defined relative to the markings of $X$ and $Y$). 

The $\cT(S)$-\textit{translation length} of a mapping class $f$, denoted $\tau_{\cT}(f)$, is defined to be 
\[ \min_{X \in \cT(S)} d_{\cT}(X, f(X)). \]
In the event that $f$ is pseudo-Anosov, one has the relation 
\[ \log(\mbox{dil}(f)) = \tau_{\cT}(f). \]

In the context of his celebrated proof of the Nielsen Realization Theorem, Kerckhoff \cite{Ker83} shows that if $\Gamma$ is a filling collection of curves, there exists a unique point $X_{\Gamma}$ in $\cT(S)$ minimizing the combined geodesic length of the curves in $\Gamma$. Roughly speaking, this argument has two steps: First, because $\Gamma$ is filling, the length function $\ell_\Gamma$ is proper, so a minimum exists. Second, one computes that the second-variation of the length function $\ell_\Gamma$ along the twist flow of a simple closed curve $\alpha$ is a sum of cosines at intersection points $\Gamma\cap \alpha$ (which are nonempty because $\Gamma$ fills $S$). Because simple closed curves are dense in the space of measured laminations, one deduces that the length function $\ell_\Gamma$ is strictly convex along any earthquake path. Because any pair of points in $\cT(S)$ is joined by an earthquake path, the minimum of $\ell_\Gamma$ is unique. Similarly, if $G \hookrightarrow S$ is a graph on $S$ that fills $S$, in the sense that each complementary region is a disk or once punctured-disk, then Kerckhoff's argument again implies that there is a unique point $X_{G} \in \cT(S)$ minimizing the length of $G$.

\subsection{The curve complex} \label{subsec: CC}

The \textit{curve complex} of $S$, denoted $\mathcal{C}(S)$, is the simplicial complex whose vertices correspond to homotopy classes of essential simple closed curves on $S$ and whose $k$-simplices correspond to collections of $k+1$ (homotopy classes of) simple closed curves that can be realized pairwise disjointly on $S$. For our purposes, it will suffice to consider the $1$-skeleton of $\mathcal{C}(S)$, known as the \textit{curve graph} (and which, slightly abusing notation, we will also denote by $\mathcal{C}(S)$). 
The curve graph is made into a metric space by identifying each edge with a unit length segment. A celebrated result of Masur-Minsky states that the associated path metric $d_{\mathcal{C}}$ on $\mathcal{C}(S)$ is $\delta$-hyperbolic \cite{MM99}, meaning that there is $\delta>0$ so that the $\delta$-neighborhood of any two of the sides of a geodesic triangle contains the third. The mapping class group $\mathcal{MCG}(S)$ acts by simplicial automorphisms on $\mathcal{C}(S)$ by acting on the homotopy classes of vertices and then extending to the higher dimensional simplices. 

There is a coarsely well-defined projection map 
\[ \mbox{sys}: \cT(S) \rightarrow \mathcal{C}(S)\] 
from the Teichm{\"u}ller space to the curve complex: given $X \in \cT(S)$, $\mbox{sys}(X)$ is sent to the simple closed curve representing the \textit{systole} of $X$, the shortest closed geodesic. The map $\mbox{sys}$ is technically not well-defined since $X$ could have multiple systoles. However, no two systoles can intersect more than once and therefore the set of all systoles constitute a diameter at most $2$ subset of $\mathcal{C}(S)$. We can then define $\mbox{sys}$ as a set map, taking values in the power set of the vertices of $\mathcal{C}(S)$.

Masur-Minsky show that $\mbox{sys}$ is \textit{coarsely Lipschitz}: there is some $K>0$ depending only on the topology of $S$ so that 
\[ d_{\mathcal{C}}(\mbox{sys}(X), \mbox{sys}(Y)) \leq K \cdot d_{\cT}(X, Y) + K. \]

Furthermore, $\mbox{sys}$ is $\mathcal{MCG}(S)$-equivariant, in the sense that $\mbox{sys}(f(X))$ coincides with $f(\mbox{sys}(X))$ for any $f \in \mathcal{MCG}(S)$ (keeping in mind that technically both are subsets of $\mathcal{C}(S)$).

\subsection{Geodesic currents} \label{subsec: GC}

In this subsection, we assume that $S$ is closed in order to avoid technical difficulties in the theory that won't arise in the course of our arguments. Fix a point $X \in \cT(S)$; it determines an action of $\pi_{1}(S)$ on the universal cover of $X$ which is identified with $\mathbb{H}^{2}$. This action extends to the Gromov boundary $\partial_{\infty}\mathbb{H}^{2} \cong S^{1}$. The \textit{space of unoriented geodesics} of $\mathbb{H}^{2}$, denoted $\mathcal{G}(\mathbb{H}^{2})$ is identified with 
\[ ((S^{1} \times S^{1}) \setminus \Delta)/\mathbb{Z}_{2}, \]
where $\Delta$ is the diagonal and the $\mathbb{Z}_{2}$ action comes from the two choices of orientation for a given bi-infinite geodesic. 

The action of $\pi_{1}(S)$ on $\partial_{\infty}(\mathbb{H}^{2})$ coming from $X$ induces a natural action of $\pi_{1}(S)$ on $\mathcal{G}(\mathbb{H}^{2})$. A \textit{geodesic current} on $X$ is then a locally finite $\pi_{1}(S)$-invariant Radon measure on $\mathcal{G}(\mathbb{H}^{2})$. 

A key example of a geodesic current comes from a choice of a (weighted) closed curve $\lambda \cdot \gamma$, for $\lambda\in\bR^+$. Given such a weighted curve, there is the counting measure supported on the set of all lifts of $\gamma$ to $\mathbb{H}^{2}$, where we assign a subset of geodesics a measure of $\lambda \cdot n$ exactly when it contains $n$ lifts. Measured laminations are also geodesic currents.

Let $C(X)$ denote the set of all geodesic currents on $X$. Celebrated work of Bonahon \cite{Bon88} shows how to naturally topologize $C(X)$ in a way that doesn't depend on the underlying choice of metric or marking, and thus one can unambiguously speak of $C(S)$, the space of geodesic currents on $S$. Moreover, the geometric intersection number of closed curves extends to a bilinear, continuous intersection form $i:C(X)\times C(X)\to \bR^+$, invariant under the natural diagonal mapping class group action by homeomorphisms. 

Bonahon also shows that $\cT(S)$ embeds in $C(S)$ in such a way where its closure reproduces Thurston's compactification by projective measured laminations. This embedding can be defined using so-called \textit{Liouville currents}: for each $X \in \cT(S)$, there is a unique geodesic current $\mathcal{L}_{X}$ such that for any closed curve $\gamma$ on $X$, 
\[ \ell_{X}(\gamma) = i(\gamma, \mathcal{L}_{X}). \]
The embedding mentioned above sends $X$ to $\mathcal{L}_{X}$. 

%Bonahon also explains how to continuously extend Thurston's intersection pairing to all geodesic currents. It thus makes sense to define a \textit{filling geodesic current}, which is one that has positive geometric intersection number with every simple closed curve. 
A geodesic current $\mu\in C(S)$ is \textit{filling} provided $i(\mu,\alpha)>0$ for every simple closed curve $\alpha$.
Of key importance for us is the following compactness criteria of Bonahon \cite{Bon88}.  

\begin{theorem} \label{compact currents} Let $c \in C(S)$ be a filling current and fix $R >0$. Then 
\[ N(c,R) = \left\{ \gamma \in C(S) : i(\gamma, c) \leq R \right\} \]
is compact in $C(S)$. 

\end{theorem}

\subsection{Counting and random walks in hyperbolic spaces} \label{subsec: R}

Let $G$ be a countable group and $\mu$ a probability distribution on $G$. We can then consider the \textit{random walk driven by} $\mu$, obtained by taking products of the form 
\[ w_{n} = g_{1}g_{2}...g_{n} \]
where $g_{i} \in G$ are independent and identically distributed with distribution $\mu$. 

Let $(X,x_{0})$ be a pointed separable $\delta$-hyperbolic space, and suppose $G$ acts by isometries on $X$. The orbit map $g \mapsto g \cdot x_{0}$ gives us a way of pushing forward a random walk on $G$ to a sequence $(w_{1}x_{0}, w_{2}x_{0},...)$ in $X$ called a \textit{sample path}. See for instance \cite{MT18} for more details. We let $\bP^{(n)}_{w}$ denote the probability induced by the $n$-fold convolution $\mu^{\ast n}$ on the space of sample paths of size $n$. For convenience and since the argument of $\bP^{(n)}_{w}$ will always be in terms of some variable that is indexed by $n$, we will often omit the superscript and will simply write $\bP_{w}$. 

The distribution $\mu$ is called \textit{non-elementary} if the subgroup of $G$ generated by its support contains a pair of loxodromic isometries with disjoint fixed points on the Gromov boundary $\partial_{\infty}X$. Maher-Tiozzo \cite{MT18} describe very detailed structural properties of sample paths in this context: 

\begin{theorem} \label{thm:MT} Let $G$ be a countable group of isometries of a separable $\delta$-hyperbolic space $(X, x_{0})$ and let $\mu$ be a non-elementary probability distribution on $G$. Then there is $L>0$ so that  
\[ \bP_{w}\left[ d_{X}(x_{0}, w_{n}x_{0}) \leq L \cdot n \right] \xrightarrow[\substack{n\to\infty}]{} 0. \]

Moreover, letting $\tau_{X}(w_{n})$ denote the translation length of $w_{n}$ on $X$, 
\[ \bP_{w}\left[ \tau_{X}(w_{n}) \leq L \cdot n \right] \xrightarrow[\substack{n\to\infty}]{} 0. \]
\end{theorem}

In fact, in this context, the progress made from the identity has an almost surely defined limiting behavior (see for instance Theorem $8.14$ of \cite{Woess02}): 

\begin{theorem} \label{thm: limit of progress exists} In the context of Theorem \ref{thm:MT}, there is $m>0$ so that 
\[ \lim_{n \rightarrow \infty} \frac{d_{X}(x_{0}, w_{n}x_{0})}{n} = m, \hspace{2 mm} \bP_{w} - \mbox{almost surely}. \]

\end{theorem}

\begin{comment}

\begin{remark} \label{drift versus translation} The arguments in Maher-Tiozzo \cite{MT18} imply that if $L>0$ is the constant so that 
\[  \bP_{w}\left[ \tau_{X}(w_{n}) \leq L \cdot n \right] \xrightarrow[\substack{n\to\infty}]{} 0, \]
then one can choose $L = m/2$. 
\end{remark}

Given a random walk driven by $\mu$ on a countable group $G$, the \textit{spectral radius} $\rho$ of the walk is defined to be
\[ \rho = \lim_{n \rightarrow \infty} \bP_{w}(w_{2n}= 1)^{1/2n}. \]
Zuk \cite{Zuk97} showed that if $\rho_{g}$ denotes the spectral radius of the symmetric random walk on the Cayley graph of a surface group with respect to a standard generating set, then 

\begin{equation} \label{Zuk}
\rho_{g} \leq \frac{1}{\sqrt{g}}
\end{equation}

In section \ref{section: lifting}, we will need the following estimate often called the Carne-Varapoulos theorem (\cite{Carne85}, \cite{Var85}) which bounds the probability of landing on any one vertex $x$ as a function of the distance to the identity $d(1,x)$

\begin{equation} \label{CV theorem}
 \bP_{w}[w_{n} = x] \leq \rho^{n} \cdot e^{-2d(1,x)^{2}/n}. 
\end{equation}

\end{comment}

Assume next that $G$ itself is Gromov hyperbolic and fix a finite generating set $\mathcal{S}$ of $G$. Let $B_{n}$ denote the ball of radius $n$ about the identity in the Cayley graph $\Gamma_{\mathcal{S}}$ associated to $\mathcal{S}$. We can then consider the probability operator $\bP^{(n)}_{g}$ coming from the uniform distribution on $B_{n}$. Again, we will sometimes drop the super-script $n$ and simply write $\bP_{g}$, but we'll make this abbreviation slightly less often than with $\bP_{w}$ since the argument of $\bP^{(n)}_{g}$ needn't involve a term that calls on $n$. 

In this setting, Gekhtman-Taylor-Tiozzo \cite{GTT18} show the analog of Theorem \ref{thm:MT}: 

\begin{theorem} \label{thm: GTT} Let $G$ be a hyperbolic group with a non-elementary action on a separable pointed hyperbolic space $(X, x_{0})$. Then the proportion of elements in $B_{n}$ that act loxodromically on $X$ goes to $1$:

\[ \frac{\#\left\{h \in B_{n}: h \hspace{2 mm} \mbox{acts loxodromically on} \hspace{2 mm} X \right\}}{\#B_{n}} = \bP^{(n)}_{g}\left[h \in G: h \hspace{2 mm} \mbox{acts loxodromically on} \hspace{2 mm} X\right] \xrightarrow[\substack{n\to\infty}]{} 1. \] 

Furthermore, a generically chosen element $h$ in $B_{n}$ moves $x$ linearly far as a function of its word length $|h|$, and in fact its stable translation length is also bounded below by a linear function of word length: there is some $L> 0$ so that 

\[ \bP^{(n)}_{g}\left[ h: d_{X}(x_{0}, h \cdot x_{0}) \geq L \cdot |h| \right] \xrightarrow[\substack{n\to\infty}]{} 1; \]
\[ \bP^{(n)}_{g}\left[h : \tau_{X}(h) \geq L \cdot |h|\right] \xrightarrow[\substack{n\to\infty}]{} 1. \]

\end{theorem}

\begin{remark} \label{stable} Gekhtman-Taylor-Tiozzo state their results in terms of \textit{stable} translation length, which is defined as 
\[ \liminf_{n \rightarrow \infty} \frac{d_{X}(x_{0}, g^n \cdot x_{0})}{n}. \]
However, it's easy to see that stable translation length is bounded above by translation length.
\end{remark}

\begin{remark} \label{rem: large length} In a surface group with a standard generating set, it is easy to see that the proportion of elements in $B_{n}$ whose word length is at least $n/2$ goes to $1$ as $n \rightarrow \infty$. It follows that when $G$ is a surface equipped with a standard generating set, one gets a linear lower bound-- in terms of $n$-- on both progress away from $x_{0}$ and  $X$-translation length for generic elements in $B_{n}$.
\end{remark}

By the phrase ``with high probability'', we will mean ``with either $\bP_{w}$ or $\bP^{(n)}_{g}$ (depending on context) going to $1$ as $n \rightarrow \infty$''. 

\section{Length minimizers} \label{section: minimizer}

For concision, we will state arguments in a way that applies to both $\bP_{g}$ and $\bP_{w}$ whenever possible. 
First, we verify that $X(\gamma_n)$ exists with high probability: %exists asymptotically almost surely as $n\to\infty$. Note that this is equivalent to the observation that $\gamma_n$ fills, by work of Kerckhoff.
\begin{lemma}
\label{lem:minimizer exists}
We have
\[
\bP_{w} \big[ \exists ! X\in \cT(S) \text{ such that } \ell_{X}(\gamma_{n}) = L(\gamma_n) \big] \xrightarrow[\substack{n\to\infty}]{} 1, 
\]
and similarly, 
\[ \bP^{(n)}_{g} \big[ \gamma \in B_{n}: \exists ! X \in \cT(S) \text{ such that } \ell_{X}(\gamma) = L(\gamma) \big] \xrightarrow[\substack{n\to\infty}]{} 1. \]
\end{lemma}

\begin{proof}
First and foremost, choose a basepoint $p \in S$. Then for $\bP_{w}$, let $w_{n} \in \pi_{1}(S, p)$ denote the location of the walk after $n$ steps, and consider the point-pushing map $f_{\gamma_{n}}$.

Since the translation length of an element $f \in \mathcal{MCG}(S,p)$ acting on the curve complex $\mathcal{C}(S)$ is larger than $2$ only if $f$ is pseudo-Anosov, it follows by Theorem \ref{thm:MT} that with high probability, $f_{w_{n}}$ is pseudo-Anosov. By Kra's theorem mentioned in Section \ref{subsec: TS}, $\gamma_{n}$ fills $S$ with high $\bP_{w}$ probability. Finally, by Kerckhoff's work mentioned in Section \ref{subsec: HG}, we are done because filling curves are length minimized at a unique point in $\cT(S)$. 

For $\bP_{g}$, the proof is the same, except instead of Theorem \ref{thm:MT}, we use Theorem \ref{thm: GTT} to argue that with high $\bP_{g}$ probability, the point-pushing map with defining curve $\gamma \in B_{n}$ is pseudo-Anosov, and therefore the homotopy class associated to the conjugacy class of $\gamma$ fills $S$. 
\end{proof}

Henceforth, we denote by $X(\gamma) \in \cT (S)$ the point which uniquely minimizes the length of $\gamma$, if it exists.

Let $\alpha$ be an arbitrary essential simple closed curve on $S$. The goal of the next lemma is to quantify the intersection between $\alpha$ and $\gamma_n$, or between $\alpha$ and a generically chosen $\gamma \in B_{n}$.

\begin{lemma} \label{linear intersection}
There exists $C > 0$ so that for any essential simple closed curve $\alpha$,
\[ \bP_{w} \left[ \frac{i(\gamma_n , \alpha)}{n} > C \right] \xrightarrow[\substack{n\to\infty}]{} 1\]

Similarly, 
\[ \bP^{(n)}_{g} \left[ \gamma \in B_{n}: \frac{i(\gamma, \alpha)}{n}> C \right] \xrightarrow[\substack{n\to\infty}]{} 1~. \]

\end{lemma}

\begin{proof}
For $\bP_{w}$, just as in in the proof of Lemma \ref{lem:minimizer exists}, we argue by reinterpreting the random walk as a walk in the point-pushing subgroup of the mapping class group of the surface obtained by puncturing $S$. Then Theorem \ref{thm:MT} implies the existence of some $L>0$ so that 
\[ \bP_{w} \left[\frac{\tau_{\mathcal{C}}(f_{w_{n}})}{n} > L \right] \xrightarrow[\substack{n\to\infty}]{} 1~,\]
where $\tau_{\mathcal{C}}(f_{w_n})$ denotes the translation length of $f_{w_{n}}$ in the curve complex of $S$. One has the analogous statement for $\bP_{g}$ by using Theorem \ref{thm: GTT}. 

Thus, the lemma follows (for instance by setting $C=L$) by establishing the following general inequality, relating translation lengths of point-pushing maps to intersection numbers between the defining closed curve and any simple closed curve: 

\[ \tau_{\mathcal{C}}(f_{\gamma}) \leq \min_{\alpha} i([\gamma], \alpha)~.  \]
 We show this using induction on intersection number. Abusing notation slightly, we will use $\gamma$ to refer both to the element in $\pi_{1}(S,p)$ and to its conjugacy class.
 
 If there exists a simple closed curve $\alpha$ disjoint from $\gamma$, it is obviously fixed by $f_\gamma$ and the desired inequality follows. If there exists $\alpha$ intersecting $\gamma $ exactly once, $\alpha$ will be disjoint from $f_{\gamma}(\alpha)$ and the inequality follows again.

In general, let $\alpha$ be a simple closed curve so that $i(\gamma, \alpha)= n$. Base $\pi_1 (S)$ at one of the intersection points $x$ between $\alpha$ and $\gamma$. Then $\gamma$ can be expressed as a concatenation of $\gamma_1 , \gamma_2 \in \pi_1 (S,x)$ as follows: starting at $x$, traverse along $\gamma$ until arriving at another intersection point $x'$ between $\gamma$ and $\alpha$; then make a choice of one of the two sub-arcs of $\alpha$ bounded by $x, x'$ to travel along to arrive back at $x$. The resulting closed curve is defined to be $\gamma_1$. Then $\gamma_2$ begins at $x$ with the same sub-arc of $\alpha$ chosen to close up $\gamma_1$ but traversed in the opposite direction. Once arriving at $x'$, traverse along the portion of $\gamma$ that $\alpha_1$ missed. It is clear that $\gamma = \alpha_1 \ast \alpha_2$ in $\pi_1(S, x)$. 

Moreover, 

\[i(\alpha, \gamma_1) \leq 1, \hspace{2 mm} \mbox{and} \hspace{2mm} i(\alpha, \gamma_2) < i(\alpha,\gamma).  \]

Thus, by the base cases demonstrated above and the induction hypothesis, we have that

\[ d_{\mathcal{C}}(\alpha, f_{\gamma_{1}}(\alpha))\leq 1, d_{\mathcal{C}}(\alpha, f_{\gamma_{2}}(\alpha)) \leq n-1, \]
where $d_{\mathcal{C}}$ denotes distance in the curve graph. Since the identification of the point-pushing subgroup with $\pi_1 (S, x)$ is induced by an isomorphism, we have that 
\[f_{\gamma}= f_{\gamma_1 \ast \gamma_2} = f_{\gamma_2} f_{\gamma_1}. \] 

We then have 
\[d_{\mathcal{C}}(\alpha, f_{\gamma}(\alpha)) = d_{\mathcal{C}}(\alpha, f_{\gamma_1 \ast \gamma_{2}}(\alpha))= d_{\mathcal{C}}(\alpha, f_{\gamma_{2}}f_{\gamma_{1}}(\alpha)) \]

\[ \leq d_{\mathcal{C}}(\alpha, f_{\gamma_{2}}(\alpha))+d_{\mathcal{C}}(f_{\gamma_{2}}(\alpha), f_{\gamma}(\alpha))\]

\[ = d_{\mathcal{C}}(\alpha, f_{\gamma_{2}}(\alpha)) + d_{\mathcal{C}}(\alpha, f_{\gamma_{1}}(\alpha)) \leq (n-1) +1 = n.\qedhere\]

\end{proof}

Given $X \in \cT (S)$, a \textit{rose} for $\cA$ centered at $x \in S$, or an $\cA$-rose, is the union of minimum length representatives for each $w_i \in \cA$, taken over all representatives starting and ending at $x$. The point $x$ is called the basepoint of the rose. Let $\ell_{\cA}(X)$ denote the infimal length of a rose on $X$, taken over all possible choices of basepoint. 

As discussed in Section \ref{subsec: HG}, Kerckhoff's argument for length functions of filling curve systems implies that there is a unique point in $\cT(S)$ minimizing the function $\ell_{\cA}: \cT (S) \rightarrow \mathbb{R}$. Denote this point by $X_{\cA}$.

We are now ready to prove the first part of Theorem \ref{thm: compact}, stated here again for convenience:

\textit{ When $S$ is closed, for almost every sample path} $\left\{w_{n}\right\}$, \textit{there is a compact set $K$ containing $X(\gamma_{n})$ for all $n$.}

We remark that in the statement of Theorem \ref{thm: compact}, the event that $X(\gamma)$ does not exist is by convention interpreted as $X(\gamma) \notin K$ for any compact set $K \subset \cT(S)$. 

\begin{proof}

Suppose there is an exhaustion of an infinite diameter neighborhood $N$ (in the Teichm{\"u}ller metric) of $X_{\cA}$ by nested compact subsets $K_0 \subset K_1 \subset K_2 \subset... $ with $X_{\cA} \in \bigcap_{i} K_{i}$ and some positive $\bP_{w}$ probability so that $X_{n} \notin K_n$.

Since $\gamma_{n}$ has geodesic length at most on the order of $n$ on the surface $X_{\cA}$, it follows that 
\[ i(w_{n}/n, \mathcal{L}_{X_{\cA}}) = O(1),\]
where $\mathcal{L}_{X_{\cA}}$ denotes the Liouville current for the surface $X_{\cA}$. Since Liouville currents are filling, Bonahon's theorem (Theorem \ref{compact currents}) implies that $\left\{w_{n}/n \right\}$ is precompact. Passing to a subsequence if necessary, we can assume the existence of a geodesic current $c$ so that 
\[ \lim_{n \rightarrow \infty}\frac{w_{n}}{n} = c. \]
We claim that with both high $\bP_{w}$ and $\bP_{g}$ probability, $c$ is a filling geodesic current. To prove this, we will show that the \textit{systole} of $c$ -- a quantity introduced by Burger-Iozzi-Parreau-Pozzetti \cite{BIPP21} and defined by 
\[ \mbox{syst}(c) = \inf_{\gamma = \hspace{2 mm} \mbox{closed geodesic}}i( \gamma, c)\]
--is positive with high probability. Amongst other things, Burger-Iozzi-Parreau-Pozzetti prove that this is equivalent to being a filling geodesic current. 

We first show that with high probability, $c$ intersects every \textit{simple} closed curve a definite number of times. Indeed, Lemma \ref{linear intersection} states that there is some $C>0$ so that with high $\bP_{g}$ and $\bP_{w}$ probability, 
\[ \frac{i(\gamma_{n}, \alpha)}{n} > C, \]
for all simple closed curves $\alpha$. Thus, for $\alpha$ an arbitrary simple closed curve, one has 
\[ i(c, \alpha) = i( \lim_{n \rightarrow \infty} w_{n}/n, \alpha) = \lim_{n \rightarrow \infty} \frac{1}{n} \cdot i(w_{n}, \alpha) \geq C. \]

Now, let $\rho$ be an arbitrary closed curve and fix some small $\epsilon>0$ with $\epsilon < C/3$. Since weighted closed curves are dense in the space of currents, there exists some weighted closed curve $\eta$ so that 
\[ |i(\eta, \kappa) - i(c, \kappa)| < \epsilon, \]
for all closed curves $\kappa$ satisfying 
\[ \ell_{X_{\cA}}(\kappa) < 10 \cdot \ell_{X_{\cA}}(\rho). \]
Then in particular, $i(\rho, \eta)$ is within $C/3$ of $i(\rho, c)$. Let $\beta$ be any essential simple closed curve obtained from $\rho$ by resolving its intersections. It follows that $\ell_{X_{\cA}}(\beta) \leq \ell_{X_{\cA}}(\rho)$, and so 
\[ i(\beta, \eta) > 2C/3. \]
Note also that this intersection number is simply the set-theoretic intersection between $\beta$ and the underlying geodesic for $\eta$, multiplied by the weight of $\eta$. 

Let $\beta'$ be a representative of $\beta$ that coincides with a collection of segments of the geodesic representative for $\rho$. Since geometric intersection is obtained by geodesic representatives, it follows that the set theoretic intersection of $\beta'$ and the underlying geodesic for $\eta$ is at least that of $\beta$ and $\eta$, and thus one has 
\[ i(\rho, \eta) > 2C/3 \Rightarrow i(\rho, c) > C/3. \]
Since $C$ does not depend on $\rho$, one then has 
\[ \mbox{syst}(c) \ge C/3, \]
as desired. 

Since $c$ is filling, it is length minimized at a unique interior point $Y$ of $\cT(S)$. Finally, the main theorem of \cite{HS21} implies that the map sending a filling current to its minimizing metric in $\cT(S)$ is continuous. Therefore, one has that 
\[ \lim_{n \rightarrow \infty} X_{n} = Y, \]
and therefore for any $i$ sufficiently large so that 
\[ d_{\cT(S)}(K_{i}, X_{\cA}) > d_{\cT(S)}(X_{\cA}, Y), \]
one must have that $K_{i} \cap  \bigcup_{n=1}^{\infty} X_{n} = \emptyset$, as desired.

\end{proof}

We next address the second part of Theorem \ref{thm: compact}: length minimizers are uniformly thick, and minimum lengths grow linearly at a definite rate with high $\bP_{w}$ and $\bP_{g}$ probability. 

For the first statement, the main tool will be Lemma \ref{linear intersection}, which we remind the reader holds for both $\bP_{w}$ and $\bP_{g}$. Now, for any word $w \in B_{n}$, a curve $\gamma$ representing $w$ has geodesic length at most some $T \cdot n$ on $X_{\cA}$, for $T$ a universal constant (for example, an upper bound on the length of any petal of length-minimizing rose for $\cA$ on $X_{\cA}$). 

Choose $\epsilon>0$ such that 
\[ \sinh^{-1} \left(\frac{1}{\sinh \left(\epsilon \right)} \right) > 2TC, \]
where $C$ is as in Lemma \ref{linear intersection}. Then for any $X \in \cT(S)$ not in $\cT_{\epsilon}(S)$, it follows that 
\[ \ell_{X}(w_{n}) > T \cdot n, \]
since $w_{n}$ has to intersect any short curve on $X$ at least $n/C$ times and must thus cross the corresponding long collar the same number of times. Similarly, the proportion of curves $\gamma \in B_{n}$ for which the above lower bound holds goes to $1$ as $n \rightarrow \infty$. It follows that with either $\bP^{(n)}_{g}$ or $\bP_{w}$ probability going to $1$ as $n \rightarrow \infty$, length minimizing metrics are $\epsilon$-thick, as desired.

For the final statement of Theorem \ref{thm: compact}, we  will phrase the argument in terms of $\bP_{w}$, but the statement for $\bP_{g}$ is proved using an identical strategy. We use the fact that the mapping class group acts cocompactly on $\cT_{\epsilon}(S)$. This will allow us to reduce to the case that $X(\gamma_{n})$ is in the mapping class group orbit of $X_{\cA}$. Indeed, there is some universal constant $D$ (depending only on $\epsilon$ and the topology of $S$) so that $X(\gamma_{n})$ is within Teichm\"{u}ller distance $D$ of some point $f \cdot X_{\cA}$ where $f \in \mathcal{MCG}(S)$. Therefore, the length of $\gamma_{n}$ on $X(\gamma_{n})$ is approximated to within a uniformly bounded multiplicative error by the minimum length $\gamma_{n}$ achieves on any point in the $\mathcal{MCG}$-orbit of $X_{\cA}$. 

Thus, the minimum length of $\gamma_{n}$ is at least some $\delta$ times the minimum length of a curve $\rho_{n}$ on $X_{\cA}$ such that $\rho_{n}$ is in the same mapping class group orbit as $\gamma_{n}$. In Section \ref{section: intersection} below, we will show that with high $\bP_{w}$ probability, the self-intersection number of $w_{n}$ is at least $n^{2}/L$ for some universal constant $L$. Letting $\rho_{n}$ denote the conjugacy class in the mapping class group orbit of $w_{n}$ with minimum possible word length, we claim that there is some $L'$ so that the word length of $\rho_{n}$ is at least $L' \cdot n$. If not, then $\rho_{n}$ has sub-linear-in-$n$ word length, and it is very easy to see that $i(\rho_{n}, \rho_{n})$ must be sub-quadratic in $n$ (see the beginning of Section \ref{section: intersection} for the simple argument). On the other hand, self-intersection is preserved over a mapping class group orbit, and so this would imply that $i(w_{n}, w_{n})$ is also sub-quadratic in $n$, which contradicts the main theorem in the following section. 

Thus, there is some $L''$ so that 
\[ \ell_{X_{\cA}}(\rho_{n}) \ge L'' \cdot n\]
\[ \Rightarrow \ell_{X(\gamma_{n})}(\gamma_{n}) \ge \delta \cdot L'' \cdot n. \]
Resetting the value of $\epsilon$ as necessary, we obtain the final statement of Theorem \ref{thm: compact}.

\section{Self-intersection number} \label{section: intersection}

In this section, we prove Theorem \ref{thm: intersection}.
To emphasize the connections between topological properties of $\gamma_{n}$ and geometric properties of its length minimizing metric $X(\gamma_{n})$, we first explain how to derive a slightly weaker result from Theorem \ref{thm: compact} when $S$ is closed. This also serves as motivation for Conjecture \ref{conj: compact} (see Remark \ref{compactness implies intersection} below for more explanation). We next present an independent argument that works for any finite type surface and which holds for $\bP_{w}$ and $\bP_{g}$.

%\begin{remark}I think this proof requires $S$ to be closed! For instance, it uses the fact that with high probability, our curves have length *at least* on the order of $n$, and this follows from the fact that they have word length at least on the order of $n$ and Milnor-Svarc. We can probably get something like $n^{2}/\log(n)^{2}$ when there are punctures by using Sisto-Taylor. 

Note that a quadratic upper bound on self intersection for $w_n$, and in fact for any $\gamma \in B_{n}$, holds non-probabilistically. Indeed, for all $\gamma\in B_n$ one has
\[i(\gamma, \gamma) \leq n\cdot \max_{k}i(w_{k}, w_{k}) + \frac{n(n-1)}{2}~ = O(n^{2}). \]
This follows by choosing a representative for $\gamma$ on an $\cA$-rose (which we will also denote by $\gamma$); $\gamma$ returns to the basepoint $n$ times and we can select a small neighborhood $N$ of the basepoint so that all self-intersection occurs within $N$, or come from self-intersection of a generator with itself in the event that the generators are not simple. Then $N \cap \gamma$ is a collection of $n$ simple arcs with endpoints on $\partial N$, and $\gamma_n$ can be homotoped within $N$ so that any two of these arcs cross at most once. 

Thus, in Theorem \ref{thm: intersection}, for both $\bP_{w}$ and $\bP_{g}$, it suffices to prove the lower bound. We first show: 

\begin{proposition} \label{closed surface intersection} When $S$ is closed, for almost every sample path $\left\{w_{n} \right\}$, there is some $L$ so that 
\[ i(\gamma_{n}, \gamma_{n}) \ge \frac{n^{2}}{L}. \]
\end{proposition}

\begin{proof}  Suppose that the lower bound in Proposition \ref{closed surface intersection} does not hold. We then have
\[ i(\gamma_{n}, \gamma_{n}) = o(n^{2}). \]

Consider, as a geodesic current, the weighted curve $\gamma_{n}/n$. Let $\mathcal{L}_{\cA}$ denote the Liouville current on $X_{\cA}$. Since $\gamma_{n}$ has length at most on the order of $n$ on $X_{\cA}$, and because the $X_{\cA}$ -length of $\gamma_{n}$ %equals the geometric intersection number between $\mathcal{L}_{p}$ and $\gamma_{n}$. 
is equal to the intersection number $i(\gamma_n,\mathcal L_{\cA})$, one has 
\[ i \left(\gamma_{n}/n, \mathcal{L}_{\cA} \right) = O(1). \]
Since $\mathcal{L}_{\cA}$ is filling, Theorem \ref{compact currents} implies that $\left\{ \gamma_{n}/ n \right\}$ lives in a compact subset of the space of geodesic currents. We can therefore pass to a convergent subsequence (and re-index the sequence if necessary) to obtain a limiting current, $c$. 

By the assumption that $i(\gamma_{n}, \gamma_{n}) = o(n^{2})$, it follows that 
\[ i(c, c) = \lim_{n \rightarrow \infty} \frac{1}{\ell_{X_{n}}(\gamma_{n})^{2}}  i(w_{n}, w_{n}) = o(1)= 0. \]
Thus, $c$ is a geodesic lamination. 

 Let $C_{fill}(S)$ denote the subset of $C(S)$ consisting of filling geodesic currents and consider the function $\pi: C_{fill}(S) \rightarrow \cT(S)$ sending a filling current $c$ to the unique point in $X(c) \in \cT(S)$ minimizing its length. This is a continuous projection (see for instance \cite{HS21}). One can extend $\pi$ to a continuous function $\overline{\pi}$ with domain $C_{fill}(S) \cup \mathcal{ML}(S)$ and codomain $\cT(S) \cup \mathcal{PML}(S)$ as follows: if $\lambda$ is a geodesic lamination, then $\overline{\pi}(\lambda) = [\lambda]$, the point in Thurston's compactification associated to the projective class of $\lambda$. 

For shorthand and to simplify notation, we will use $X_{n}$ to denote $X(\gamma_{n})$ in what follows. 
Continuity of $\overline{\pi}$ implies that 
\[ \lim_{n \rightarrow \infty} \overline{\pi}(\gamma_{n}/\ell_{X_{n}(\gamma_{n})}) = \overline{\pi}(\lim_{n \rightarrow \infty}\gamma_{n}/\ell_{X_{n}(\gamma_{n})}),\]
whenever both sides exist. If $c$ is a geodesic lamination, the left hand side is a point in $\mathcal{PML}(S)$, whereas by Theorem \ref{thm: compact}, the right hand side is a point in the interior of $\cT(S)$. This is a contradiction, and therefore $c$ can not be a geodesic lamination, and hence $i(\gamma_{n}, \gamma_{n})$ can not grow sub-quadratically. This completes the proof.

\end{proof}

\begin{remark} \label{compactness implies intersection} We note that affirming Conjecture \ref{conj: compact} would- in combination with the proof of Proposition \ref{closed surface intersection}- yield a complete proof of Theorem \ref{thm: intersection} for closed surfaces. Indeed, if one knew that $X(\gamma_{n})$ stayed within distance $D$ of $X_{\cA}$ with high probability, one could bound the self-intersection number of the limiting current $c$ uniformly from below, which would in turn imply a uniform (over almost every sample path) upper bound on the constant $L$ in the statement of Proposition \ref{closed surface intersection}.
\end{remark}

 The proof of Theorem \ref{thm: intersection} uses the assumption that $S$ is closed in fairly essential ways. For example, when $S$ has punctures, it is easy to see that Theorem \ref{compact currents} is false. In that setting, the complement of a filling curve $\gamma$ will contain punctured disks; one can then choose a sequence $(\gamma_{i})$ of closed curves each intersecting $\gamma$ some bounded number of times, but which self intersect more and more by winding more times about a cusp in the complement of $\gamma$. Continuity of self-intersection implies that $N(c,R)$ is not compact: the sequence $(\gamma_i)$ cannot have a convergent subsequence. 
 
Moreover, the conclusion of Proposition \ref{closed surface intersection} is not independent of the choice of sample path; a priori, given any $\epsilon>0$, there could still be a positive $\bP_{w}$ probability that $i(\gamma_{n}, \gamma_{n})$ is less than $\epsilon \cdot n^{2}$. Therefore, to obtain the full power of Theorem \ref{thm: intersection}, we need a different argument.

\vspace{2 mm}

\textit{Proof of Theorem \ref{thm: intersection}:} 

\vspace{2 mm}

One starts with Lemma \ref{linear intersection} to say that with high $\bP_{g}$ and $\bP_{w}$ probability, $\gamma_{n}$ intersects every simple closed curve at least $n/C$ times. On the other hand, in \cite{S16} it is shown that there is some constant $J= J(S)$ depending only on the topology of the underlying surface $S$ so that if $\gamma$ is any closed curve with $i(\gamma, \gamma) \leq K$, then there is some essential simple closed curve $\alpha$ satisfying 
\[ i(\alpha, \gamma) <  J \cdot \sqrt{K}. \]
Thus, we deduce that if 
\[ i(\gamma_{n}, \gamma_{n}) < \frac{n^{2}}{C^{2}J^{2}} \]  with positive probability, there must exist a simple closed curve $\alpha_{n}$ intersecting $\gamma_{n}$ less than $n/C$ times with positive probability, which contradicts Lemma \ref{linear intersection}. \hfill $\Box$

\section{Lifting degree} \label{section: lifting}

In this section, we prove Theorems \ref{thm: simple lifting degree for punctures} and \ref{thm: simple lifting degree for closed surfaces}. The upper bound on simple lifting degree in both settings is straightforward using work of Patel \cite{Patel14} and the fact that given a fixed hyperbolic surface $X \in \cT(S)$, any $\gamma \in B_{n}$ admits a representative with $X$-length at most $K \cdot n$, where $K$ depends only on $X$. It follows that $\gamma \in B_{n}$ has length at most $K \cdot n$ on the surface used by Patel to bound simple lifting degrees from above, and thus $\deg(\gamma) \leq K \cdot 17 \cdot n$. 

%Using joint work of the authors with Patel and Sapir, we obtain the following (non-probabilistic) upper bound on the minimum degree of a cover required so that $\gamma_{n}$ admits a simple elevation: 

%\begin{corollary} \label{Lifting degree upper bound} There is $K>0$ so that 

%\[ \ \min \left\{ d : \exists \hspace{2 mm} \mbox{covering map} \hspace{2 mm} \pi: Y \rightarrow S \hspace{2 mm} \mbox{of degree} \hspace{2 mm} d \hspace{2 mm} \mbox{s.t.} \hspace{2 mm} \gamma_{n} \hspace{2 mm} \mbox{admits a simple elevation to} \hspace{2 mm} Y \right\} \leq K \cdot n . \]

%\end{corollary}

%\begin{proof} The self-intersection number of $\gamma_{n}$ is bounded above by a quadratic function of $n$. 
%Aougab-Gaster-Patel-Sapir show that the minimum degree of a cover $\pi: Y \rightarrow S$ required so that a given curve $\gamma$ admits a simple elevation to $Y$ is bounded above by a constant times the square root of self-intersection
%Aougab-Gaster-Patel-Sapir show that the degree of $\gamma_n$ is bounded above by a constant times the square root of the self-intersection of $\gamma_n$, times a constant that depends on the distance from $X_n$ to a special point in $\cT(S)$.
%By Lemma~\ref{lem:compact}, this implies the stated bound.
%\end{proof} 

We next tackle the lower bound on simple lifting degree for $\bP_{g}$ in the case of punctures or boundary components, and for a generic conjugacy class in the setting of closed surfaces. The argument will be simpler to organize by replacing each puncture with a boundary component; this has no effect on the relevant counts. By $S_{g,b}$, we will mean the surface of genus $g$ with $b \geq 0$ boundary components. 

\begin{theorem} \label{thm: generic lifting} When $\pi_{1}(S)$ is free, there is some $\mathcal{M}>0$ so that with high $\bP_{g}$ probability, the simple lifting degree is at least $\mathcal{M} \cdot n/\log(n)$. Moreover, when $S$ is closed, 

\[ \lim_{n \rightarrow \infty} \frac{ \#([c] \in C_{n}: \deg([c]) \geq \mathcal{M} \cdot n/log(n))}{\#(C_{n})} \rightarrow 1. \]

\end{theorem}

\begin{proof} Let $d$ denote the simple lifting degree $\deg(\gamma)$ of some curve $\gamma \in B_{n}$. A theorem of Hall \cite{Hall49} yields the following recursive count for $N_{d}(r)$, the number of conjugacy classes of index $d$ subgroups of $F_{r}$, the free group of rank $r$:

\[ N_{d}(r) = d \cdot (d!)^{r-1} - \sum_{k=1}^{d-1}(d-k)!^{r-1} \cdot N_{k}(r) \leq (d+1)!^{r} \]

For surface groups, we get the following count for (what we will call) $N_{d}(2g)$ due to Mednykh (see for instance \cite[Ch.~14]{LS03}): 

\[  N_{d}(2g)= \frac{h_{d}}{(d-1)!} - \sum_{k=1}^{d-1}\frac{h_{d-k}\cdot N_{k}(2g)}{(d-k)!}~, \]
where $h_{d}$ is the quantity
\[ h_{d}= (d!)^{(2g-1)}\cdot \sum_{\chi \in \text{Irr}(\Sigma_{d})} \chi(1)^{2-2g}~. \]
In the above sum, $\chi$ ranges over all irreducible characters of $\Sigma_{d}$, the symmetric group on $d$ symbols. 
 Irreducible representations of $\Sigma_{d}$ are in correspondence with conjugacy classes of $\Sigma_{d}$, and thus the number of irreducible representations is equal to the number of partitions of $\left\{1,..., d \right\}$. Therefore, using the classical asymptotic expansion for the number of such partitions and the fact that the degree of an irreducible representation must divide the order of the group, we get that

\[ \frac{N_{d}(2g)}{(d!)^{2g+1}\cdot e^{\pi \sqrt{2d/3}}} = O(1),   \]
from whence it follows that for $d$ sufficiently large, 
\[ N_{d}(2g) \leq M \cdot (d!)^{2g+2} \]
for some universal constant $M>0$. 

Fix $c>0$ and consider the unique hyperbolic surface $\mathfrak{p}$ homeomorphic to a pair of pants with all boundary components having geodesic length $c$. Then, fix once and for all a choice of hyperbolic metric on $S$ with totally geodesic boundary that comes from gluing together the appropriate number of copies of $\mathfrak{p}$ in some configuration that produces a surface homeomorphic to $S$. By construction, this hyperbolic surface (which, abusing notation slightly, we will also call $S$) admits a pants decomposition $\mathcal{P}= \left\{\gamma_{1},..., \gamma_{m} \right\}$ where each curve has geodesic length $c$ (and where $m= m(S)$ depends only on the topology of $S$). It follows by the Milnor-Svarc lemma that $\gamma$ has geodesic length at most $\tau \cdot n$ for some universal constant $\tau$ depending only on $S$.

Let $\rho: Y \rightarrow S$ be a covering space of degree at most $d$. The choice of metric on $S$ pulls back to $Y$- again by a slight abuse of notation, refer to the resulting hyperbolic surface by $Y$. Let $\mathcal{P}_{Y}$ denote the full pre-image of $\mathcal{P}$ under $\rho$ to $Y$.

We will first specify a pants decomposition on $Y$, equipped with this metric, such that each curve has length bounded above polynomially in terms only of $c,d$.

\begin{lemma} \label{lem: good pants on Y} There is an explicit polynomial $\mathcal{F}(x,y)$ of degree $3$ in $x$ and degree $4$ in $y$ satisfying the following. There exists a pants decomposition $P_{Y}$ of $Y$ containing $\mathcal{P}_{Y}$ as a sub-multi-curve such that each curve in $P_{Y}$ has geodesic length at most $\mathcal{F}(d\cdot |\chi(S)|, d \cdot c)$.

\end{lemma}

\begin{proof}

 Each curve in $\mathcal{P}_{Y}$ has geodesic length at most $d \cdot c$. Note that $\mathcal{P}_{Y}$ is a simple multi-curve, but it need not be a pants decomposition on $Y$. In the event that $\mathcal{P}_{Y}$ is not a pants decomposition, there exists complementary components of $S \setminus \mathcal{P}_{Y}$ which are not pairs of pants; let $Z \subset Y$ be such a subsurface. Then $Z$ is bounded by curves in $\mathcal{P}_{Y}$. 

There is then a $1$-Lipschitz map from a hyperbolic surface (with boundary) $W$ to $Z$, obtained by perhaps identifying distinct boundary components of $W$. The surface $W$ has area at most $\text{Area}(Y)= d \cdot \text{Area}(S)$, and each boundary component of $W$ has length at most $d \cdot c$. 

Using the polynomial $\mathcal{F}(x,y)$ mentioned at the end of Section \ref{subsec: HG}, we see that $W$ admits a pants decomposition of total length at most $\mathcal{F}(d \cdot |\chi(S)|, d \cdot c)$. We can push this pants decomposition forward to $Z$, and apply this to each non-pants subsurface in the complement of $\mathcal{P}_{Y}$ to obtain a complete pants decomposition of $Y$ such that each curve has length at most $\mathcal{F}(d \cdot |\chi(S)|, d \cdot c)$. \qedhere

\end{proof}

\begin{remark} \label{not too short} Note that there is a universal constant $\kappa$, depending only on $S$, such that each curve in $P_{Y}$ has geodesic length at least $\kappa$. Indeed, one can set $\kappa$ equal to the length of the systole of $S$. 

\end{remark}

Now, assume $\gamma \subset S$ is a closed curve that admits a simple elevation $\tilde{\gamma}$ to $Y$. Our goal will be to bound the Dehn-Thurston coordinates of $\tilde{\gamma}$ with respect to $P_{Y}$. To begin, by Remark \ref{not too short}, each twisting number is bounded above by $2\kappa \cdot d \cdot \tau \cdot n$, for if $\tilde{\gamma}$ twists more than this number of times about any curve in $P_{Y}$, it will be too long. 

It remains to bound the intersection number $i(\tilde{\gamma}, w)$ for each $w \in P_{Y}$. There are two possible cases, depending on whether or not $w$ is in the pre-image of $\mathcal{P}$. 

\begin{lemma} \label{lem: w comes from S} There is a universal constant $\eta>0$ such that if $w \in \mathcal{P}_{Y}$, then 
\[ i(\tilde{\gamma}, w) \leq d \cdot \eta \cdot n. \]
\end{lemma}

\begin{proof} Recall that the geodesic length of $\gamma$ on $S$ is at most $\tau \cdot n$ for some universal constant $\tau$. We can decompose $\gamma$ into sub-arcs bounded by intersection points with $\mathcal{P}$ and containing no such intersection points in their interiors. Thus each sub-arc lies entirely inside of one pair of pants, each of which is a copy of $\mathfrak{p}$. If we then set $\eta$ to be the product of $\tau$ with the minimum distance between any two boundary components of $\mathfrak{p}$, we see that 
\[ i(\gamma, \rho(w)) \leq \eta \cdot n. \]

\end{proof}

If, on the other hand, $w$ is not in the pre-image of $\mathcal{P}$, then $w$ comes from the construction outlined in Lemma \ref{lem: good pants on Y}, and $\rho(w)$ lies entirely inside of one copy of $\mathfrak{p}$ on $S$. It follows that 
\[ \ell_{S}(\rho(w)) \leq \mathcal{F}(d \cdot |\chi(S)|, d \cdot c). \]

\begin{comment} 

For simplicity of notation we will abbreviate the right hand side above by $\mathcal{F}$. It follows from work of Basmajian \cite{Bas13} that 
\[ i(\rho(w), \rho(w)) \leq H \cdot \mathcal{F}^{2}, \]
where $H$ is a universal constant depending only on $\mathfrak{p}$. 

\end{comment} 

We wish to control the number of intersections between $\rho(w)$ and $\gamma$; we will then get a bound on the intersection number between $w$ and $\tilde{\gamma}$ by multiplying by $d$. 

%To do this, partition $\rho(w)$ into at most $H \cdot \mathcal{F}^{2}$ maximal embedded sub-arcs $\left\{\lambda_{1},..., \lambda_{k} \right\}$. We will bound the intersection number between any such $\lambda_{i}$ and $\gamma$. 

Decompose $\mathfrak{p}$ into two right-angled hyperbolic hexagons $h_{1}, h_{2}$ by cutting along geodesic arcs orthogonal to the boundary components. Note that each sub-arc of $\rho(w) \cap h_{i}$ ($i=1,2$) is an embedded arc with endpoints on distinct edges of $\partial h_{i}$ not arising from $\partial \mathfrak{p}$. It follows that the number of such sub-arcs is bounded above by $(c/2) \cdot \mathcal{F}$. Indeed, $\rho(w)$ has length at most $\mathcal{F}$, and each edge of $h_{i}$ coming from $\partial \mathfrak{p}$ has length $c/2$ and it constitutes the shortest path between the two edges on either side of it. Since $\gamma$ has geodesic length at most $\tau \cdot n$, by the same argument, the number of components of $\gamma \cap h_{1}$ is at most $(c/2) \cdot \tau \cdot n$. 

Since a given sub-arc of $\gamma \cap h_{i}$ can intersect one of $\rho(w) \cap h_{i}$ ($i= 1,2$) at most once, we get that 

\[ i(\rho(w), \gamma) \leq \frac{c^{2} \cdot \tau \cdot n \cdot \mathcal{F}}{4} \Rightarrow i(w, \tilde{\gamma}) \leq \frac{d \cdot c^{2} \cdot \tau \cdot n \cdot \mathcal{F}}{4}. \]

Hence, each Dehn-Thurston coordinate of $\tilde{\gamma}$ with respect to $P_{Y}$ is bounded above by 
\[ 2 \kappa \cdot d \cdot \tau \cdot n + d \cdot \eta \cdot n + \frac{d \cdot c^{2} \cdot \tau \cdot n \cdot \mathcal{F}}{4} =: Q(d),\]
which we note depends polynomially on $d$ of degree at most $5$. The number of simple closed (multi-)curves on $Y$ with all Dehn-Thurston coordinates bounded above by $Q(d)$ is at most the number of integer points in a cube in $\mathbb{R}^{\dim(\cT(Y))}$ with side lengths bounded by $Q(d)$, which of course is 
\[ Q(d)^{d \cdot \dim(\cT(S))}. \]
It follows that the number of homotopy classes of curves on $S$ with simple lifting degree at most $d$ is bounded above by 
\[ M \cdot (d+1)!^{2g+p+2} \cdot Q(d)^{d \cdot \dim(\cT(S))} \]
\[ \sim M \cdot \left( \frac{(d+1)}{e} \right)^{(d+1)(2g+p+2)} \cdot Q(d)^{d \cdot \dim(\cT(S))}. \]
 Taking the logarithm of the above produces 
 \[ (d+1)(2g+p+2) \cdot \log \left( M \cdot (d+1) \right) + d \cdot \dim(\cT(S)) \cdot \log(Q(d))-(d+1)(2g+p+2). \]
Now, if $d = o(n/log(n))$, the above grows sub-linearly in $n$.

At this stage, we have shown the existence of some $\mathcal{M}$ such that the number of distinct conjugacy classes $[c]$ with a representative $c \in B_{n}$ having lifting degree less than $\mathcal{M} \cdot n/log(n)$, grows sub-exponentially in $n$. This completes the proof of the lower bound for generic conjugacy classes for closed surfaces, and thus for Theorem \ref{thm: simple lifting degree for closed surfaces}; indeed, $|C_{n}|$ grows exponentially with the same growth rate of the Cayley graph itself (see for instance \cite{CK02}). 

It therefore remains to prove the lower bound for a generically chosen \textit{element} in the free group setting. Theorem \ref{thm: generic lifting} will be implied by the following lemma which bounds the size of an intersection of any conjugacy class with $B_{n}$ above by a function that grows at most exponentially with a strictly smaller growth rate than $|B_{n}|$ itself.

\begin{lemma} \label{lem: bounding conjugacy} Let $B_{n}$ denote the ball of radius $n$ in the Cayley graph of a free group $F_{r}$ with respect to a free basis. Then 
\[ \#(B_{n} \cap [c]) \leq n \cdot \#(B_{(n-||c||)/2}), \]
where $[c]$ is any conjugacy class and $||c||$ is the minimum word length for any representative of $[c]$. 

\end{lemma}

\begin{remark} After constructing the following argument, the authors found a very similar and more concise proof of Lemma \ref{lem: bounding conjugacy} in Parkkonen-Paulin \cite{PP15}. We include our original version here for self-containment. 
\end{remark}

\begin{proof}  Let $c \in F_{r}$ be a cyclically reduced representative of the conjugacy class $[c]$; $c$ is also a minimal length representative.

Let 
\[ c= s_{1}s_{2}...s_{p} \]
be a minimal length spelling of $c$. By an \textit{initial subword} of $c$, we will mean any word of the form $s_{1}...s_{k}$ for $k \leq p$. Similarly, a \textit{tail subword} of $c$ will be any word of the form $s_{k}...s_{p}$ for $k \geq 1$. Say that a word $w \in F_{r}$ satisfies the \textit{no-cancellation condition} with respect to $c$ if $w$ satisfies both of the following: 
\begin{itemize}
    \item No tail subword of $w$ coincides with the inverse of an initial subword of $c$;
    \item No tail subword of $w$ coincides with a tail subword of $c$. 
\end{itemize}

If $w$ satisfies the no-cancellation condition with respect to $c$, then the word length of $wcw^{-1}$ is precisely $2\mbox{length}(w)+ \mbox{length}(c)$. It follows that the number of words $w$ that satisfy the no-cancellation condition and that have the property that $wcw^{-1}$ has word length at most $n$, is at most the number of words with word length $(n-||c||)/2$, i.e., $\#(B_{(n-||c||)/2})$. 

On the other hand, if $x \in F_{r}$ does not satisfy the no-cancellation condition, there is some initial subword $x'$ of $x$ that does satisfy the no-cancellation condition, and which is obtained from $x$ by deleting the tail subword that experiences cancellation with $c$. We will refer to $x'$ as the \textit{no-cancellation reduction} of $x$. Up to multiplication by a power of $c$, to each word $x'$ satisfying the no-cancellation condition, there are at most $\mbox{length}(c)$ words whose no-cancellation reduction is $x$: these correspond to multiplying $x'$ on the right by all possible tail subwords of $c$. 

Since $wc^{k}$ and $w$ conjugate $c$ to the same element, it follows that for every conjugator $w$ satisfying the no-cancellation condition, there are at most $\mbox{length}(c) \leq n$ many conjugates coming coming from conjugators that do not satisfy the no-cancellation condition and whose no-cancellation reduction is $w$. Thus the total number of conjugates of $c$ with length at most $n$ is at most $n \cdot \#(B_{(n-||c||)/2})$, as desired.

\end{proof}

We now complete the proof of Theorem \ref{thm: generic lifting}: letting $J(n)$ denote the number of conjugacy classes with minimal word length at most $n$ and with the property that the simple lifting degree is less than $\mathcal{M} \cdot n/\log(n)$, Lemma \ref{lem: bounding conjugacy} implies that the total number of elements in $B_{n}$ with lifting degree less than $\mathcal{M} \cdot n/log(n)$ is at most 
\[ T \cdot J(n) \cdot n \cdot \#(B_{n/2}). \]
Since $J(n)$ grows sub-exponentially, we have that 
\[ \lim_{n \rightarrow 0} \frac{J(n) \cdot n \cdot \#(B_{n/2})}{\#(B_{n})} = 0, \]
as desired.

\end{proof}

We next turn our attention to $\bP_{w}$ in the setting of punctures. We  assume that $\mu$ is supported on a minimal rank symmetric generating set and that it assigns equal weight to each generator. In this setting, the $\bP_{w}^{(n)}$ probability of a given event is simply the proportion of length $n$ sample paths terminating at an element described by that event. When the Cayley graph is a tree, there is a one-to-one correspondence between elements and non-backtracking paths originating at the identity; we will use this to derive the desired result for $\bP_{w}$ from the result above regarding $\bP_{g}$ in the presence of punctures or boundary components.  

\begin{remark} \label{PS versus hitting} It would be interesting to study $\bP_{w}$ in the setting of more general distributions $\mu$. Our strategy is to understand $\bP_{w}$ from Theorem \ref{thm: generic lifting} and the properties of $\bP_{g}$. The challenge of extending this strategy so that it applies to more general choices of $\mu$ relates to the difficulty of comparing the hitting measure for a random walk with the Patterson-Sullivan measure on the boundary of a hyperbolic group $G$-- see for instance Proposition 1.7 of \cite{GTT18}. 
\end{remark}

\begin{proposition} \label{prop: lifting degree for walks with punctures} With high $\bP_{w}$ probability, $w(n)$ has simple lifting degree at least (on the order of) $n/log(n)$ when $S$ has a puncture. \end{proposition}

\begin{proof}  
 
Under the assumptions made on $\mu$, the $\bP^{(n)}_{w}$ probability that $w_{n}$ equals some given $x \in \pi_{1}(S)$ is equal to the proportion of length $n$ sample paths terminating at $x$. Let $T^{(n)}(x)$ denote the number of length $n$ sample paths terminating at $x$. Since the Cayley graph is a tree, there is a unique non-backtracking path between the identity and any such $x$-- denote this path by $p(x)$. 

We claim that $T^{(n)}(x)$ is a function only of the distance $d(x)$ between $x$ and the identity. For example if $d(x)= n$, $T^{(n)}(x)= 1$. If $d(x) = n- C$ for some $C>0$, any sample path terminating at $x$ must be comprised of a family of loops with total length $C$ appended to $p(x)$. Given any loop family (defined to be a finite collection of loops based at potentially different vertices) in the Cayley graph, by its \textit{type}, we will mean the result of translating each loop in the family to the identity. There are only finitely many types of loop families with length $C$; an arbitrary sample path contributing to $T^{(n)}_{x}$ is given by choosing one such family, and for each loop $\ell$ in that family, a choice of vertex along $p(x)$ at which it is based. Let then $T^{k}$ denote the size of $T^{k}(x)$ for any $x \in S_{k}$, the sphere of radius $k$ about the identity in the Cayley graph. 

From the proof of Theorem \ref{thm: generic lifting}, there is a constant $K>0$ so that if $J(n)$ denotes the number of elements in $B_{n}$ with lifting degree less than $\mathcal{M} \cdot n/\log(n)$ then 
\[ \lim_{n \rightarrow \infty} \frac{J(n)}{\#(B_{\lambda \cdot n/2})} = 0, \]
for all $\lambda> 1$. On the other hand, if $H(n)$ denotes the number of elements in $S_{n}$ with lifting degree \textit{at least} $\mathcal{M} \cdot n/\log(n)$, $H(n)$ grows on the order of $\#(B_{n})$ (indeed, $\#(S_{n})$ grows exponentially with the same growth rate as $\#(B_{n})$). 

It follows that the number of length $n$ sample paths terminating at an element in $S_{k}$ with lifting degree less than $\mathcal{M} \cdot n/\log(n)$ is at most 
\[ J(k) \cdot T^{k}. \]

\begin{comment}
Let $\mathcal{A}_{n}$ denote the annulus 
\[ \mathcal{A}_{n} = \left\{ x \in \pi_{1}(S): L \cdot n \leq d(x) \leq n \right\}. \]
We can then decompose $\mathcal{A}_{n}$ into spheres $\mathcal{S}_{L \cdot n},..., \mathcal{S}_{n}$, where 
\[ \mathcal{S}_{k} = \left\{x \in \pi_{1}(S): d(x) = k \right\}, \]
and by the previous paragraph, $T^{(n)}$ is constant on each sphere. 

Moreover, the number of elements in each such sphere grows exponentially in $n$, where the coefficient of $n$ is bounded below in terms only of $L$. Thus, by the argument used in the proof of Theorem \ref{thm: generic lifting}, as $n \rightarrow \infty$, the proportion of elements in each sphere with lifting degree at least $\frac{1}{\mathcal{M}} \cdot n/\log(n)$, goes to $1$. 

\end{comment}

Now, fix some small $\epsilon>0$. By Theorem \ref{thm: limit of progress exists}, there is some function $f(n)$ satisfying $\lim_{n \rightarrow \infty} f(n)/n = 0$ so that for large enough $n$, a fraction of at least $1- \epsilon$ of all length $n$ sample paths terminate on some sphere $\mathcal{S}_{k}$ for $m \cdot n - f(n) \leq k \leq m \cdot n + f(n)$. It follows that, if $\mathcal{P}_{n}$ denotes the subset of length $n$ sample paths terminating on such a sphere, then $\mathcal{P}_{n}$ accounts for at least a proportion of $1-\epsilon$ of all length $n$ sample paths, and the number of length $n$ sample paths with lifting degree less than $\mathcal{M} \cdot n/\log(n)$ in $\mathcal{P}_{n}$ is at most

\[ J(m \cdot n - f(n)) \cdot T^{m \cdot n- f(n)} +...+ J(m \cdot n +f(n)) \cdot T^{m \cdot n + f(n)} =: T(n). \]

On the other hand, the number of length $n$ sample paths in $\mathcal{P}_{n}$ terminating at a point with lifting degree at least $\mathcal{M} \cdot n / \log(n)$ is
\[ H(m \cdot n- f(n)) \cdot T^{m \cdot n- f(n)} +... + H(m \cdot n + f(n)) \cdot T^{m \cdot n + f(n)} =: G(n). \]
We claim that 
\[ \lim_{n \rightarrow \infty}\frac{T(n)}{G(n)} = 0. \]
If this is so, then it would follow that for $n$ sufficiently large, the above ratio is at most $\epsilon$, and therefore that the proportion of length $n$ sample paths with lifting degree less than $\mathcal{M} \cdot n/\log(n)$ is at most $\epsilon + \epsilon \cdot (1-\epsilon)$. Let $\epsilon \rightarrow 0$ and we are done. 

It therefore remains to prove that $G(n)$ overwhelms $T(n)$ as $n \rightarrow \infty$. Fixing $n$, choose $k, m \cdot n - f(n) \leq k \leq m \cdot n + f(n)$ maximizing the function $T^{k}$. Then 
\[ T(n) \leq 2f(n) \cdot J(m \cdot n + f(n)) \cdot T^{k}, \hspace{2 mm} \mbox{and} \hspace{2 mm} G(n) \geq H(k) \cdot T^{k}.\]
Therefore, 
\[ \frac{T(n)}{G(n)} \leq \frac{2f(n) \cdot J( m \cdot n +f(n)) \cdot T^{k}}{H(k) \cdot T^{k}} = \frac{2f(n) \cdot J(m \cdot n +f(n))}{H(k)}. \]
Since $k \geq m \cdot n- f(n)$, the denominator grows exponentially in $n$ with growth rate at least that of $\#(B_{m \cdot n})$. On the other hand, the growth rate of the numerator is at most that of $\#(B_{m \cdot n/2})$. This completes the proof.

\end{proof}

We conclude this section with a weaker version of the $\bP_{w}$ piece of Theorem \ref{thm: simple lifting degree for punctures} that holds for a much wider class of generating sets.

Theorem $1.1$ of Sisto-Taylor \cite{ST19}  states that if $Q$ is an infinite index quasiconvex subgroup of a hyperbolic group $G$ and $\mathcal{S}$ is a fixed finite generating set containing a generating set for $Q$, then if $\bP_{w}$ denotes the probability operator induced by a random walk on the Cayley graph $\Gamma_{\mathcal{S}}$ driven by some $\mu$ with finite support generating $G$, there is some $C>0$ so that
\[ \bP_{w}( \max_{xQ \in G/Q} d_{xQ}(\pi_{xQ}1, \pi_{xQ}w_{n})  \in [C^{-1}\log(n), C \log(n)]) \rightarrow 1. \]
Here $\pi_{xQ}$ denotes the nearest point projection to the left coset $xQ$ (well-defined because $Q$-- and therefore $xQ$-- is quasiconvex).

In our context, $G$ will be $\pi_{1}(S)$ for $S$ a surface with non-empty boundary and $Q$ will be the infinite cyclic subgroup associated to a boundary component. Choose a generating set for $\pi_{1}(S)$ containing a generator for the infinite cyclic subgroup associated to the boundary component. We will identify $\mbox{Cay}_{S}(G)$ with its image under its orbit map $\mathcal{O}$ into the universal cover $\mathcal{U}$ which, after choosing a complete hyperbolic metric on $S$ with totally geodesic boundary, we can associate with a subset of $\mathbb{H}^{2}$. 

With this interpretation, the work of Sisto-Taylor \cite{ST19} implies that the the geodesic segment from $\mathcal{O}(1)$ to $\mathcal{O}(w_{n})$ projects to a path of length on the order of $\log(n)$ to the geodesic in $\mathbb{H}^{2}$ projecting to $\partial S$. 
This implies that there is an annular cover $A$ of $S$ and a lift of $w_n:S^1\to S$ to $\widetilde w_n:\bR\to A$ so that both ends of $\widetilde w_n$ lie on one of the two ends of $A$, and so that $\widetilde{w_{n}}$ self-crosses on the order of $\log(n)$ times. Lower bounds on the degree of $w_n$ can now be made accessible via the following lemma, motivated by the work in \cite{Gaster16}:

\begin{lemma}
\label{lem:degree}
Suppose that $\pi:A\to S$ is an annular cover, and $\widetilde{\gamma}$ is a lift of the curve $\gamma:S^1\to S$ to $\widetilde \gamma: \bR\to A$ so that both ends of $\widetilde \gamma$ lie on a common end of $A$. 
Then $\deg\gamma\ge \iota(\widetilde{\gamma},\widetilde{\gamma})$.
\end{lemma}

\begin{proof}
Suppose that $p':S'\to S$ is a finite-sheeted cover of $S$ so that $\gamma$ has a simple elevation $\gamma'\subset S'$, and so that $\deg p'=\deg \gamma$.
Let $\alpha\subset S$ be the image of the core curve of $A$. Because $p'$ is a finite-sheeted cover, there is a smallest integer $r\ge 1$ so that $\alpha^r$ admits an elevation $\widetilde \alpha$ to $S'$ that is `next to' $\gamma'$.

Consider the degree $r$ covering $A_r\to A$, and let $\widetilde \gamma_r \subset A_r$ be a lift of $\widetilde \gamma$.
By the lifting criterion, the map $A_r\to S$ admits a lift to $A_r\to S'$ so that the core curve of $A_r$ maps to $\widetilde \alpha$, and under which $\widetilde\gamma_r$ maps to $\gamma'$.
Because $\gamma'$ is simple, $\widetilde \gamma_r$ is simple as well. 
Observe that the self-intersection number $m:=\iota(\widetilde \gamma,\widetilde \gamma)$ measures the number of times $\widetilde \gamma$ winds around the core of $A$. 
Because $\widetilde \gamma_r$ is simple (i.e.~embedded), we must have $r\ge m$.

It is evident that $\deg p' \ge r$, as points on $\alpha$ have at least $r$ preimages under $p'$. Therefore $\deg p' \ge m$ as claimed.
\end{proof}

It follows that when $S$ has a boundary component $b$ and we choose a finite generating set $\mathcal{S}$ containing a generator for the infinite cyclic subgroup coming from $b$, then if $\mu$ is \textit{any} distribution with support $\mathcal{S}$ and $\bP_{w}$ is the probability operator corresponding to a random walk on $\Cay_{\mathcal{S}}$ driven by $\mu$, there is $K$ so that
\[ \lim_{n \rightarrow \infty} \bP_{w} \left[ \deg(w_{n}) \geq K \cdot \log(n) \right] \rightarrow 1.\]

\section{Pushing-point dilatation of random curves} \label{section: point push}

We conclude the paper with a brief proof of Theorem \ref{thm: point push}. 

\begin{proof} For $\bP_{w}$, Theorem \ref{thm:MT} implies that the curve complex translation length $\tau_{\mathcal{C}}(f_{w_{n}})$ is larger than $L \cdot n$ with high $\bP_{w}$ probability. It then  follows from the coarse-Lipschitzness and $\mathcal{MCG}$-equivariance of the systole map $\mbox{sys}$ (see Section \ref{subsec: CC}) that the Teichm{\"u}ller translation length $\tau_{\cT}(f_{w_{n}})$ is at least some $(L/K) \cdot n$ with high $\bP_{w}$ probability. 

Since Teichm{\"u}ller translation length is nothing but the log of dilatation, one then has 
\[ \mbox{dil}(f_{w_{n}}) \geq e^{(L/K) \cdot n}\]
with high $\bP_{w}$ probability. 

On the other hand, there is some $J>0$ so that 
\[ \mbox{dil}(f_{w_{n}}) \leq e^{J \cdot n}, \]
just because the orbit map of $\mathcal{MCG}(S)$ into $\cT(S)$ is (coarsely) Lipschitz and $w_{n}$ has word length at most $n$. 

On the other hand, Theorem \ref{thm: intersection} gives the existence of some $L'\ge 1$ so that (with high $\bP_{w}$-probability),
\[ \frac{n^{2}}{L'} \leq i(\gamma_{n}, \gamma_{n}) \leq L' \cdot n^{2}. \]

It follows that with high $\bP_{w}$ probability, one has 
\[ e^{(L/K) \cdot \sqrt{i(\gamma_{n}, \gamma_{n})/L'}} \leq \mbox{dil}(f_{w_{n}}) \leq e^{J \cdot \sqrt{L' \cdot i(\gamma_{n}, \gamma_{n}) }}. \]
For $\bP_{g}$, the proof is identical, using Theorem \ref{thm: GTT} in place of Theorem \ref{thm:MT}, and using the $\bP_{g}$ portion of Theorem \ref{thm: intersection}.

\end{proof}

\bibliographystyle{acm}
\bibliography{main}

\end{document}